\documentclass[leqno,a4paper]{article}
\usepackage{amsmath,amsthm,amssymb}
\usepackage[utf8]{inputenc}
\usepackage[T1]{fontenc}

\usepackage{enumerate}
\usepackage{bbm}
\usepackage{todonotes}
\usepackage{mathrsfs}
\usepackage{mathabx}
\usepackage[breaklinks]{hyperref}
\usepackage{nicefrac}
\usepackage{tikz}
\usetikzlibrary{matrix}

\usepackage{url}
\makeatletter
\g@addto@macro{\UrlBreaks}{\UrlOrds}
\makeatother
\def\UrlBreaks{\do\/\do-}
\usepackage{breakurl}

\usepackage[sort,numbers]{natbib}

\providecommand{\noopsort}[1]{} 

\def\qed{\unskip\quad \hbox{\vrule\vbox
to 6pt {\hrule width 4pt\vfill\hrule}\vrule} }

\newcommand{\bez}{\nopagebreak\hspace*{\fill}
 \nolinebreak$\qed$\vspace{5mm}\par}

\newtheorem{Th}{Theorem}[section]
\newtheorem{Prop}[Th]{Proposition}
\newtheorem{Lemma}[Th]{Lemma}

\theoremstyle{definition}
\newtheorem{Remark}[Th]{Remark}
\newtheorem{Def}{Definition}[section]
\newtheorem{Cor}[Th]{Corollary}
\newtheorem{Example}{Example}[section]

\newcommand{\beq}{\begin{equation}}
\newcommand{\eeq}{\end{equation}}

\def\scalar(#1,#2){(#1\mid#2)}
\newcommand{\til}{\widetilde}
\renewcommand{\hat}{\widehat}

\newcommand{\ca}{{\cal A}}
\newcommand{\cb}{{\cal B}}
\newcommand{\cc}{{\cal C}}
\newcommand{\cd}{{\cal D}}

\newcommand{\cm}{{\cal M}}

\newcommand{\xbm}{(X,{\cal B},\mu)}
\newcommand{\zdr}{(Z,{\cal D},\rho)}
\newcommand{\ycn}{(Y,{\cal C},\nu)}
\newcommand{\zdk}{(Z,{\cal D},\kappa)}
\newcommand{\ot}{\otimes}
\newcommand{\ov}{\overline}
\newcommand{\la}{\lambda}

\newcommand{\bs}{\mathbb{S}}
\newcommand{\A}{\mathbb{A}}
\newcommand{\B}{\mathbb{B}}

\newcommand{\C}{{\mathbb{C}}}
\newcommand{\Z}{{\mathbb{Z}}}
\newcommand{\N}{{\mathbb{N}}}
\newcommand{\E}{{\mathbb{E}}}

\newcommand{\va}{\varphi}

\newcommand{\mob}{\boldsymbol{\mu}}
\newcommand{\lio}{\boldsymbol{\lambda}}

\newcommand{\bfu}{\boldsymbol{u}}
\newcommand{\bfv}{\boldsymbol{v}}

\begin{document}

\title{Automatic sequences are orthogonal to  aperiodic multiplicative functions}
\author{Mariusz Lema\'nczyk\thanks{Research supported by a Narodowe Centrum Nauki grant.
} \and Clemens M\"ullner\thanks{Research supported by European Research Council (ERC) under the European Union’s Horizon 2020 research and innovation programme under the Grant Agreement No 648132, by the project F55-02 of the Austrian Science Fund FWF which is part of the
Special Research Program “Quasi-Monte Carlo Methods: Theory and Applications” and by Project I1751 (FWF), called MUDERA (Multiplicativity, Determinism, and Randomness).}
}

\maketitle

\thispagestyle{empty}
\begin{abstract}Given a finite alphabet $\A$ and a primitive substitution    $\theta:\A\to\A^\lambda$ (of constant length $\la$), let $(X_\theta,S)$ denote the corresponding dynamical system, where $X_\theta$ is the closure of the orbit via the left shift $S$ of a fixed point of the natural extension of $\theta$ to a self-map of $\A^{\Z}$. The main result of the paper is that all continuous observables in $X_\theta$ are orthogonal to any bounded, aperiodic, multiplicative function $\bfu:\N\to\C$, i.e.
\[  \lim_{N\to\infty}\frac1N\sum_{n\leq N}f(S^nx)\bfu(n)=0\]
for all $f\in C(X_\theta)$ and $x\in X_\theta$. In particular, each primitive automatic sequence, that is, a sequence read by a primitive finite automaton, is orthogonal to any bounded, aperiodic, multiplicative function.
\end{abstract}

\section*{Introduction} Throughout the paper, by an {\em automatic}
sequence $(a_n)_{n\geq0}\subset\C$, we mean a continuous
observable in a substitutional system $(X_\theta,S)$, i.e.\
$a_n=f(S^nx)$, $n\geq0$, for some $f\in C(X_\theta)$ and
$x\in X_\theta$.\footnote{This is a slight extension of the
classical notion of automatic sequence meant as a sequence
obtained by a finite code of a fixed point of a substitution,
see e.g.\ \cite{Al-Sh}, \cite{Fo},  \cite{Qu} to see also a relationship with sequences read by finite automata.} Here, we
assume that $\theta:\A\to \A^\lambda$ is a substitution
of constant length $\la$ (over the alphabet $\A$), and we
let $(X_\theta,S)$  denote the corresponding subshift of
the full shift $(\A^{\Z},S)$ (see Section~\ref{s:subst} for details).

By $\mob:\N\to\{-1,0,1\}$ we denote the classical M\"obius function: $\mob(p_1\ldots p_k)=(-1)^k$ for different primes $p_1,\ldots,p_k$, $\mob(1)=1$ and $\mob(n)=0$ for all non square-free numbers $n\in\N$. In connection with the celebrated Sarnak's conjecture \cite{Sa}
on M\"obius orthogonality of zero entropy systems, i.e.
\beq\label{mdis}
\lim_{N\to\infty}\frac1N\sum_{n\leq N}f(S^nx)\mob(n)=0\eeq
for each zero entropy dynamical system $(X,S)$,
all $f\in C(X)$ and $x\in X$
(for more information on the subject, see the survey article
\cite{Fe-Ku-Le}), it has been proved in \cite{Mu} that
all automatic sequences $(a_n)$ are orthogonal to the classical
M\"obius function $\mob$, i.e.\
$\lim_{N\to\infty}\frac1N\sum_{n\leq N}a_n\mob(n)=0$. This triggers the question whether~\eqref{mdis} remains true if we replace $\mob$ by another arithmetic function.
The M\"obius function is an example of an arithmetic function
which is multiplicative ($\mob(mn)=\mob(m)\mob(n)$ whenever
$m,n$ are coprime), hence, it is natural to ask whether
automatic sequences are orthogonal to each zero mean,\footnote{Recall that a sequence $\bfu:\N\to\C$ has
{\em zero mean} if $M(\bfu):=\lim_{N\to\infty}\frac1N\sum_{n\leq N}
\bfu(n)$ exists and equals zero.} bounded, multiplicative function $\bfu$.
Said that, one realizes immediately  that the answer to such a question is
negative as periodic functions are automatic sequences and  we have
many periodic, multiplicative functions.\footnote{Indeed, examples of
periodic multiplicative functions are given by: Dirichlet characters,
or $n\mapsto(-1)^{n+1}$.} But even if we consider the non-periodic
case, still, there are non-periodic automatic sequences which are
(completely) multiplicative, zero mean functions  \cite{Al-Go},
\cite{Sch-P}, \cite{Ya}.  On the other hand, it has been proved in
\cite{Fe-Ku-Le-Ma} that many automatic sequences given by so called
bijective substitutions are orthogonal to all zero mean, bounded,  multiplicative
functions (in \cite{Dr}, it is proved that they are orthogonal to
the M\"obius function).

A stronger requirement than zero mean of $\bfu$ which one can consider
is that of {\em aperiodicity}, that is
$$
\lim_{N\to\infty}\frac1N\sum_{n\leq N}\bfu(an+b)=0$$
for each $a,b\in\N$, i.e.\ $\bfu$ has a mean, equal to zero, along each arithmetic progression.
Many classical multiplicative functions are aperiodic,
e.g.\ $\mob$ or the Liouville function $\lio$.

The aim of the present paper is to prove the following ($h(\theta)$ and $c(\theta)$ stand, respectively,  for the height and the column number of the substitution $\theta$, see Section~\ref{s:subst}):

\begin{Th}\label{t:main} Let $\theta$ be a  primitive
substitution of constant length $\la$. Then, each automatic
sequence $a_n=f(S^nx)$, $n\geq0$, in $(X_\theta,S)$ is orthogonal to any bounded,
aperiodic, multiplicative function $\bfu:\N\to\C$,  i.e.
\beq\label{orto}
\lim_{N\to\infty}\frac1N\sum_{n\leq N}a_n\bfu(n)=0.\eeq
More precisely:

(i) If  $c(\theta)=h(\theta)$ then each automatic sequence
$(f(S^nx))_{n\geq0}$ is orthogonal to any bounded, aperiodic,
arithmetic function $\bfu$.\footnote{We emphasize that no
multiplicativity on $\bfu$ is required.}

(ii) If $c(\theta)>h(\theta)$ then:
the automatic sequences $(f(S^nx))_{n\geq0}$ for which the spectral
measure of $f$ is continuous are orthogonal to all bounded,
multiplicative
functions. If the spectral measure is discrete then (i) applies.
All other automatic sequences in $(X_\theta,S)$ are orthogonal to
all bounded, aperiodic, multiplicative  functions.
\end{Th}

We have already mentioned that examples of automatic sequences which are multiplicative functions are known, but they are quite special. For example, in \cite{Sch-P}, it is proved that completely multiplicative, never vanishing functions ``produced'' by finite automata are limits in the upper density of periodic sequences, that is, they are Besicovitch rationally almost periodic. In fact, we can strengthen this result by showing that they are even Weyl rationally almost periodic,\footnote{A sequence $(b_n)$ taking values in a finite set $\B$ is called {\em Weyl rationally almost periodic} if it can be approximated by periodic sequences (with values in $\B$) in the pseudo-metric
$$d_W(x,y) = \limsup_{N\to\infty}\sup_{\ell\geq1}\frac1N|\{\ell\leq n<\ell+N:\: x(n)\neq y(n)\}|.$$} which is a consequence of Theorem~\ref{t:main}:

\begin{Cor}\label{c:main} All multiplicative and automatic sequences produced by  primitive automata\footnote{Sequences produced by finite automata take only finitely many values and are, therefore, bounded.} are Weyl rationally almost periodic.
Furthermore, the automaton/substitution that produces such an automatic sequence can be chosen to have equal column number and height.
\end{Cor}

The main problem in this paper belongs to number theory and
combinatorics, and  the proof of  M\"obius orthogonality for
automatic sequences in \cite{Mu} relied on
combinatorial properties of sequences produced by automata and
an application of the (number theoretic) method of Mauduit and Rivat
\cite{Ma-Ri}, \cite{Ma-Ri1}. However, the problem of orthogonality of sequences
is deeply related to classical ergodic theory, namely, to
Furstenberg's disjointness of dynamical systems, see \cite{Fe-Ku-Le}
for an exhaustive presentation of this approach. A use of ergodic
theory tools  is the main approach in the present paper: we make use
of some old results on the centralizer of substitutional systems
\cite{Ho-Pa} and, more surprisingly, we find an ergodic interpretation
of the combinatorial approach from \cite{Mu} in terms of joinings of
substitutional systems. Finally, we relativize some of the arguments
from \cite{Fe-Ku-Le-Ma} to reach the goal. For example, the reason
behind the orthogonality of automatic sequences to all bounded,
multiplicative functions
(whenever the spectral measure of an automatic sequence is continuous)
in (ii) of Theorem~\ref{t:main}  is that the essential centralizer
of substitutional systems is finite \cite{Ho-Pa}. This ``small size''
of the essential centralizer puts
serious restrictions on possible joinings between
(sufficiently large) different prime powers $S^p$, $S^q$ of $S$
and allows  one to
use the numerical  KBSZ criterion on the orthogonality of numerical
sequences with bounded, multiplicative functions (see Section~\ref{s:kbszcrit}).

The ergodic theory approach (together with number-theoretic tools)
turn out to be very effective in some attempts to prove Sarnak's conjecture. The recent results of
Frantzikinakis and Host \cite{Fr-Ho1} show: if instead of
orthogonality~\eqref{mdis}, we ask for its weaker, namely, logarithmic
version:
\beq\label{mdislog}
\lim_{N\to\infty}\frac1{\log N}\sum_{n\leq N}\frac1n f(S^nx)\mob(n)=0
\eeq
then this orthogonality indeed holds for all zero entropy $(X,S)$ systems whose set
of ergodic (invariant) measures is countable, in particular, it applies to
substitutional systems. Moreover, in \eqref{mdislog}, we can replace
$\mob$ by many other so called non-pretentious multiplicative functions
\cite{Fr-Ho2}. However, so far, the approach through logarithmic
averaging does not seem to answer the main conjecture~\eqref{mdis}
(see also \cite{Go-Kw-Le} and \cite{Ta2017}).

The ergodic theory approach allows one for a further extension of the notion of being automatic sequence which triggers one more natural question. Namely, a uniquely ergodic topological dynamical system $(Y,T)$ (with a unique invariant measure $\nu$) is called MT-{\em substitutional} if there is a primitive substitution $\theta:\A\to\A^{\lambda}$ such that the measure-theoretic dynamical systems $(Y,\nu, T)$ and $(X_{\theta},\mu_{\theta},S)$ are measure-theoretically isomorphic. Then, any sequence $a_n=g(T^ny)$, $n\geq0$ (with $g\in C(Y)$ and $y\in Y$), can be called MT-{\em automatic}.\footnote{For example, sequences given by pieces of automatic sequences in $(X_\theta,S)$: $g(S^nx_k)$ for $b_k\leq n<b_{k+1}$, $k\geq1$, where $b_1=1$ and $b_{k+1}-b_k\to\infty$, for $g\in C(X_\theta)$ and $(x_k)\subset X_\theta$, are MT-automatic, see Section~\ref{s:lastS}.}  To cope with the MT-automatic case we need to control a behaviour of $\bfu$ on so called short intervals:
\beq\label{zalo2}
\lim_{K\to\infty}\frac1{b_K}\sum_{k<K}\left|
\sum_{b_k\leq n<b_{k+1}}\bfu(n)\right|=0\eeq
(for each $(b_k)$ satisfying $b_{k+1}-b_k\to\infty$) which is much stronger than the requirement that $M(\bfu)=0$. It turns out however that~\eqref{zalo2} is satisfied for many so called non-pretentious multiplicative functions, see \cite{Ma-Ra}, \cite{Ma-Ra-Ta}; in particular, it is satisfied for the M\"obius function $\mob$.

\begin{Cor}\label{c:main2} Any MT-automatic sequence $(a_n)$ satisfies~\eqref{orto}
for any multiplicative, bounded, aperiodic $\bfu:\N\to\C$ satisfying~\eqref{zalo2}.  For each MT-substitutional system $(Y,T)$ and any $g\in C(Y)$, we have
$$
\frac1N\sum_{n\leq N}g(T^ny)\bfu(n)\to0\text{ when }N\to\infty$$
{\bf uniformly} in $y\in Y$. In particular, the uniform convergence takes place in $(X_\theta,S)$ for any primitive substitution $\theta$.
\end{Cor}

In the course of the proof of Theorem~\ref{t:main}, given a
substitutional system $(X_\theta,S)$, we will build
successive continuous extensions of it, which are also given by
substitutions, where for the largest extension, Theorem~\ref{t:main}
will be easier to handle. As a byproduct of this procedure, we will clear up Remark 9.1 from \cite{Qu}, p.~229, about
the form of a cocycle in a skew product representation of $(X_\theta,\mu_\theta,S)$ over the
Kronecker factor.

{\bf Added in June 2018:} In the recent paper \cite{Ko}, it is proved that all $q$-multiplicative sequences are either almost periodic which roughly would correspond to~(i) of Theorem~\ref{t:main} or are orthogonal to all bounded, multiplicative functions. This makes some overlap (some automatic sequences are $q$-multiplicative) with our main result but the classes considered in \cite{Ko} and in the present paper are essentially different.

\section{Ergodic theory necessities}
\subsection{Joinings and disjointness}
By a (measure-theoretic) {\em dynamical system} we mean $(X,\cb,\mu,S)$, where $\xbm$
is a probability standard Borel space and $S:X\to X$ is an a.e.\
bijection which is bimeasurable and measure-preserving.
If no confusion arises, we will speak about $S$ itself and call
it an {\em automorphism}.\footnote{In what follows we will
also use notation $S\in{\rm Aut}\xbm$,
where ${\rm Aut}\xbm$ stands for the Polish group of
all automorphisms of $\xbm$. The topology is given by
the strong operator topology of the corresponding unitary
operators $U_S$, $U_Sf:=f\circ S$ on $L^2\xbm$.}

\begin{Remark}\label{r1} Each homeomorphism $S$ of
a compact metric space $X$ determines many (measure-theoretic) dynamical systems
$(X,\cb(X),\mu,S)$ with $\mu\in M(X,S)$, where $M(X,S)$ stands
for the set of probability Borel measures on $X$ ($\mathcal{B}(X)$
stands for the $\sigma$-algebra of Borel sets).
Recall that by the Krylov-Bogolyubov theorem, $M(X,S)\neq\emptyset$,
and moreover, $M(X,S)$ endowed with the weak-$\ast$ topology becomes a
compact metrizable space. The set $M(X,S)$ has a natural structure of
a convex set (in fact, it is a Choquet simplex) and its extremal
points are precisely the ergodic measures.
We say that the topological system $(X,S)$ is {\em uniquely ergodic}
if  it has only one invariant measure (which must be ergodic).
The system $(X,S)$ is called {\em minimal} if it does not contain a
proper subsystem (equivalently, the orbit of each point is dense).
\end{Remark}

\begin{Remark}\label{r:podszifty} Basic systems considered in the paper
are {\em subshifts} whose definition we now recall.
Let $\A$  be  a finite, nonempty set (alphabet). By a {\em block}
(or {\em word}) over $\A$, we mean $B\in\A^n$ (for some $n\geq0$) and $n=:|B|$ is the {\em length} of $B$.
Hence $B=(a_{i_0},a_{i_1},\ldots,a_{i_{n-1}})$ with $a_{i_k}\in \A$ for
$k=0,\ldots,n-1$. We will also use the following notation:
$B=B_0B_1\ldots B_{n-1}$, where $B_j=a_{i_j}$ for $j=0,\ldots,n-1$.  If $0\leq i\leq j<n$ then we write $B[i,j]$ for
$B_iB_{i+1}\ldots B_j$, in particular, notationally, $B[i]=B_i$.
We say that the (sub)block $B[i,j]$ {\em appears} in $B$. The
notation we have just presented has its natural extension to infinite
sequences.

Given $\eta\in\A^{\N}$, we can define
$$
X_\eta:=\{x\in \A^{\Z}:\mbox{each block appearing in $x$ appears in
$\eta$}\}.$$
It is not hard to see that $X_\eta$ is closed and $S$-invariant, where
$S:\A^{\Z}\to\A^{\Z}$ is the left {\em shift}, i.e.
$$
S((x_n)_{n\in\Z})=(y_n)_{n\in\Z},\text{ where }y_n=x_{n+1},\;n\in\Z.$$
Then the dynamical system $(X_\eta,S)$ is called a {\em subshift}
(given by $\eta$).
If $\eta$ has the property that each block appearing in it
reappears infinitely often (and such are substitutional systems which
are considered in the paper)
then there exists $\ov{\eta}\in\A^{\Z}$ satisfying:
$\ov{\eta}_k=\eta_k$ for each $k\geq0$ and
$X_\eta=\ov{\{S^m\ov{\eta}:\:m\in\Z\}}$.
\end{Remark}

Given another system $(Y,\cc,\nu,T)$, we may consider the set
$J(S,T)$ of {\em joinings} of automorphisms $S$ and $T$. Namely,
$\kappa\in J(S,T)$ if $\kappa$ is an $S\times T$-invariant
probability measure on $\cb\ot\cc$ with the projections
$\mu$ and $\nu$ on $X$ and $Y$, respectively.\footnote{We can also
speak about (topological) joinings in the context of topological
dynamics. Indeed, if $(X,S), (Y,T)$ are two topological systems
then each closed subset $M\subset X\times Y$ invariant under
$S\times T$ and with full projections is called a (topological)
joining.} Note that the projections maps
$p_X:X\times Y\to X$, $p_Y:X\times Y\to Y$ settle factor
maps between the dynamical systems
$$
(X\times Y,\cb\ot\cc,\kappa,S\times T)\text{ and } (X,\cb,\mu,S),\:(Y,\cc,\nu,T),\text{ respectively}.$$
The automorphisms $S$ and $T$ are called {\em disjoint}
if the only joining of $S$ and $T$ is product measure
$\mu\ot\nu$, i.e. $J(S,T)=\{\mu\ot\nu\}$. We will then write
$S\perp T$.
Note that if $S\perp T$ then at least one of these
automorphisms must be ergodic.
If both are ergodic, the subset $J^e(S,T)$ of ergodic
joinings (i.e.\ of those $\rho\in J(S,T)$ for which the system
$(X\times Y,\rho,S\times T)$ is ergodic) is non-empty; in fact,
the  ergodic decomposition of a joining consists (a.e.) of ergodic
joinings.

If $(X,\cb,\mu,S)$ and $(Y,\cc,\nu,T)$ are isomorphic, i.e.\ for some (invertible\footnote{Without non-invertibility, we say that $T$ is a {\em factor} of $S$.}) $W:\xbm\to\ycn$, we have equivariance $W\circ S=T\circ W$, then $W$ yields the corresponding {\em graph joining} $\mu_W\in J(S,T)$ determined by
$$
\mu_W(B\times C)=\mu(B\cap W^{-1}C).$$
Then, $(X\times Y,\cb\ot\cc,\mu_W,S\times T)$ is isomorphic
to $S$ (hence $\mu_W$ is ergodic if $S$ was).
When $S=T$, we speak about {\em self-joinings} of $S$.

We recall that an automorphism $R\in{\rm Aut}\zdk$ has {\em discrete spectrum} if $L^2\zdk$ is spanned by the eigenfunctions of the unitary operator $U_R$: $f\mapsto f\circ R$ on $L^2\zdk$. Assuming ergodicity, we have:
\beq\label{jdis}
\mbox{If $R$ has discrete spectrum then each its ergodic self-joining is graphic.}\eeq

Each automorphism $S\in{\rm Aut}\xbm$ has the largest factor which has discrete spectrum.  It is called the {\em Kronecker factor} of $S$.

Joinings are also considered in topological dynamics. If $S_i$ is a homeomorphisms of a compact metric space $X_i$, $i=1,2$ then each $S_1\times S_2$-invariant, closed subset $M\subset X_1\times X_2$ with the full natural projections, is called a {\em topological joining} of $S_1$ and $S_2$.
\subsection{Discrete suspensions} Given $S\in {\rm Aut}\xbm$ and $h\in\N$, consider the probability space $(\widetilde{X},\widetilde{\mu})$, where $\widetilde{X}=X\times\{0,1,\ldots,h-1\}$ and $\widetilde{\mu}=\mu\ot \rho_h$, where $\rho_h$ is the normalized uniform measure on $\Z/h\Z \cong \{0,1,\ldots,h-1\}$. We define now $\widetilde{S}$ on $(\widetilde{X},\widetilde{\rho})$ by setting
\beq\label{susp11}
\widetilde{S}(x,j)=(x,j+1)\text{ whenever }x\in X,\; j=0,1,\ldots,h-2\eeq
and
\beq\label{susp12}
\widetilde{S}(x,h-1)=(Sx,0)\text{ for }x\in X.\eeq
It is not hard to see that $\widetilde{S}\in{\rm Aut}(\widetilde{X},\widetilde{\mu})$ and it is called the $h$-{\em discrete suspension} of $S$ (it is ergodic if and only if so is $S$).

Note that the map $(x,j)\mapsto j$ for $(x,j)\in\widetilde{X}$ yields a factor map between $\widetilde{S}$ and the rotation $\tau_h: x\mapsto x+1$ on $\Z/h\Z$. Note also that $\widetilde{S}^h(x,0)=(Sx,0)$, so in fact we can view $\widetilde{S}$ as the $h$-discrete suspension of $\widetilde{S}^h|_{X\times\{0\}}$. As a matter of fact, an automorphism $R\in {\rm Aut}\zdk$ is an $h$-discrete suspension if and only if
\beq\label{susp14}\mbox{ $R$ has the rotation $\tau_h$ as a factor.}\eeq
Indeed, if $\pi$ settles a factor map between $R$ and $\tau_h$, then set $A_0:=\pi^{-1}(\{0\})$, note that $R^h(A_0)=A_0$ and show that $R$ is isomorphic to the $h$-discrete suspension of $R^h|_{A_0}$.

Finally, consider $(m,h)=1$. Then, it is easy to see that:
\beq\label{susp15}
\mbox{The $h$-discrete suspension $\widetilde{\tau_m}$ is isomorphic to direct product $\tau_{m}\times\tau_h$.}\eeq

We refer the reader to \cite{Gl} for more information on ergodic theory, in particular, on the theory of joinings.

\subsection{Group and isometric extensions}
Given $(X,\cb,\mu,S)$ an ergodic dynamical system, consider a
measurable $\va:X\to G$ (i.e.\ $\varphi$ is a {\em cocycle}), where $G$ is a compact metric group.
Let $m_G$ denote Haar measure of $G$.
The automorphism $S_\varphi:X\times G\to X\times G$, given by
$$
S_\varphi(x,g)=(Sx,\varphi(x)g),$$
is called a {\em compact group extension} of $S$. We obtain the
dynamical system $(X\times G,\cb\ot\cb(G),\mu\ot m_G,S_\varphi)$
which need not be ergodic. For example, it is not ergodic when
$\varphi(x)=\xi(Sx)^{-1}\xi(x)$   for a measurable
$\xi:X\to G$, i.e.\ when $\varphi$ is a {\em coboundary} (indeed, the map $(x,g)\mapsto(x,\xi(x)g)$ settles an isomorphism between $S_\varphi$ and $S\times Id_G$).
Note that $\left(S_\varphi\right)^m(x,g)=(S^mx,\va^{(m)}(x)g)$, where
$\va^{(m)}(x)=\va(S^{m-1}x)\ldots\va(Sx)\va(x)$ for $m\geq1$.

\begin{Prop}[\cite{Fu}]\label{p1}
Assume that $S$ and $T$ are ergodic automorphisms on $\xbm$ and
$\ycn$, respectively. Assume moreover that $S\perp T$ and
$S_\varphi$, $T_\psi$ are ergodic group extensions of $S$ and $T$,
respectively ($\psi:Y\to H$). If the product measure
$(\mu\ot m_G)\ot(\nu\ot m_H)$ is ergodic\footnote{It means that the only situation in which we loose
disjointness is when the cocycles $\va$ and $\psi$ ``add''
a common non-trivial eigenvalue for $U_{S_\varphi}$ and $U_{T_\psi}$.}
then $S_\varphi\perp T_\psi$.
\end{Prop}

The following results are also classical.

\begin{Lemma}\label{l3} Assume that $S_\varphi$ and $T_\psi$ are
ergodic and let $\widetilde\rho\in J^e(S_\varphi,T_\psi)$. Then
(up to a natural permutation of coordinates)
$(X\times G\times Y\times H,\cb\ot\cb(G)\ot\cc\ot\cb(H),
\widetilde{\rho},S_\varphi\times T_\psi)$ is a compact
group extension of $(X\times Y,\cb\ot\cc,\rho,S\times T)$,
where $\rho:=\widetilde{\rho}|_{X\times Y}$.

Moreover, if the relatively independent extension\footnote{This measure is defined by $$\widehat{\rho}(\widetilde{A}\times\widetilde{B})=
\int_X\E(\widetilde{A}|X)(x)\E(\widetilde{B}|X)(y)\,d\rho(x,y)\text{ for Borel subsets }\widetilde{A},\widetilde{B}\subset X\times G.$$} $\widehat{\rho}$ of $\rho$ is ergodic then $\widetilde{\rho}=\widehat{\rho}$.
\end{Lemma}

A group extension is a special case of so called {\em skew products}.
Assume that we have a measurable map $\Sigma:X\to{\rm Aut}\,\zdr$.
Then
we can consider $S_\Sigma$:
$$
S_\Sigma(x,z):=(Sx,\Sigma(x)(z))$$
which
is an automorphism of $(X\times Z,\cb\ot\cd,\mu\ot\rho)$.
If additionally, $Z$ is a compact metric space and $\Sigma(x)$, $x\in X$, are
isometric then we call $S_\Sigma$ an {\em isometric extension}.
Since the group ${\rm Iso}(Z)$ of isometries of $Z$ considered with the
uniform topology is a compact metric group
(by Arzela-Ascoli theorem), it is not hard to see that each
isometric extension is a factor of a group extension
(especially, if we assume that the isometric extension is
uniquely ergodic).\footnote{A group extension can be defined on
$X\times {\rm Iso}(Z)$, acting by the formula $(x,I)\mapsto
(Sx,\Sigma(x)\circ I)$; to obtain a factor map, fix a point  (transitive for ${\rm Iso}(Z)$) $z_0\in Z$ and
consider $(x,I)\mapsto (x,I(z_0))$.}
  If the isometric extension is ergodic, one can
choose the group extension also ergodic.

\begin{Remark}\label{r2} If $S_\varphi$ is a group extension, then for
each closed subgroup $F\subset G$, the automorphism $S_{\varphi,F}$,
given by
$$S_{\varphi,F}(x,gF):=(Sx,\varphi(x)gF),$$
is an isometric extension of $S$; it acts on the space
$X\times G/F$ considered with $\mu\ot m_{G/F}$, where
the measure $m_{G/F}$ on the homogenous space $G/F$ is
the natural image of Haar measure $m_G$. \end{Remark}

\begin{Remark}\label{r:efi} Each finite extension is isometric. It is a
factor of a group extension by $G$, where $G$ is finite.\end{Remark}

\subsection{Odometers}
{\em Odometers} are given by the inverse limits of cyclic groups:
we have $n_{t-1}|n_{t}$ for each $t\geq1$ and $$X=H_{(n_t/n_{t-1})}:={\rm liminv}\, \Z/n_t\Z$$ with the rotation $R$ by $1$ on each coordinate. If for each $t$, $n_{t+1}/n_t=\la\geq2$, then we speak about $\la$-{\em odometer} and denote it by $H_\lambda(=H_{(\lambda)})$.

Each odometer $(X,R)$ is uniquely ergodic (with the unique measure being Haar measure $m_X$ of $X$). Then $(X,m_X,R)$ has discrete spectrum with the group of eigenvalues given by all roots of unity of degree $n_t$, $t\geq1$.
Furthermore, $R^r$ is ergodic (uniquely ergodic) iff $(r,n_t)=1$ for each $t\geq1$. In this case $R^r$ and $R$ are isomorphic as both are ergodic and their spectra are the same, so the claim follows by the Halmos-von~Neumann theorem (e.g.\ \cite{Gl}).
It easily follows that whenever $p,q\in\mathscr{P}$ are different prime numbers (by $\mathscr{P}$ we denote the set of prime numbers) not dividing any $n_t$ then each $\rho\in J^e(R^p,R^q)$ is a graph joining (of an isomorphism between $R^p$ and $R^q$), see e.g.\ \cite{Le-Me1}.

\subsubsection{$h$-discrete suspensions of odometers}
Assume that $(h,n_t)=1$ for each $t\geq1$. Let $\widetilde{R}$ denote the $h$-discrete suspension of $R$.  Then (cf.~\eqref{susp15}), we obtain that
\beq\label{suspodom}
\mbox{$\widetilde{R}$ is isomorphic to $R\times\tau_h$.}\eeq
Indeed, both automorphisms have discrete spectrum (and are ergodic). The group of eigenvalues of $U_{\widetilde{R}}$ is equal to $\{e^{2\pi ij/(h n_t)}:\:j\in\Z,t\geq0\}$, while the group of eigenvalues of $U_{R\times\tau_h}$ is generated by $\{e^{2\pi ij/n_t}:\:j\in\Z,t\geq0\}$ and the group of $h$-roots of unity. It follows that $U_{R\times\tau_h}$ and $U_{\widetilde{R}}$ have the same group of eigenvalues, hence again by the Halmos-von Neumann theorem, they are isomorphic.
\subsubsection{Group extensions of odometers - special assumptions} We will assume that $R$    is an odometer with ``small spectrum'', that is, the set $\{p\in\mathscr{P}:\:p|n_t\text{ for some }t\geq1\}$ is finite (this will always be the case for $\la$-odometers). Moreover, we will assume that $R_\varphi$ (more precisely, $U_{R_\varphi}$) has continuous spectrum on the space $L^2(X\times G,m_X\ot m_G)\ominus L^2(X,m_X)$.~\footnote{If $V$ is a unitary operator on a Hilbert space $\mathscr{H}$ then it is said to have {\em continuous spectrum} if for all
$x\in \mathscr{H}$ the spectral measure $\sigma_x$ is continuous;
the latter measure (on $\bs^1$) is determined by its Fourier
transform: $\widehat{\sigma}_x(m):=\int_{\bs^1}z^m\,d\sigma_x(z)=
\langle V^mx,x\rangle$ for all $m\in\Z$.}   It follows that
if $p\in\mathscr{P}$ is sufficiently large then $(R_\varphi)^p$
is ergodic   (uniquely ergodic if $\varphi$ is continuous).
Furthermore, we assume that if $p,q\in\mathscr{P}$ are
different and sufficiently large then the only ergodic joinings
between $(R_\varphi)^p$ and $(R_\varphi)^q$ are relatively
independent extensions of isomorphisms between
$R^p$ and $R^q$.\footnote{We recall that if $W:X\to X$ settles
an isomorphism of $R^p$ and $R^q$, then it yields an ergodic
joining $(m_X)_W(A\times B):=m_X(A\cap W^{-1}B)$;
its relatively independent extension $\widehat{(m_X)}_W$ is
defined by
$$
\widehat{(m_X)}_W(\widetilde{A}\times\widetilde{B})=
\int_X\E(\widetilde{A}|X)(x)\E(\widetilde{B}|X)(Wx)\,dm_X(x)\text{ for Borel subsets }\widetilde{A},\widetilde{B}\subset X\times G.$$}
By all these assumptions $(X\times G\times X\times G,
\widehat{(m_X)}_W, (R_\varphi)^p\times(R_\varphi)^q)$ is
ergodic and has the same eigenvalues as the odometer $R$.

\begin{Lemma}\label{l4} Assume that $\varphi$ is continuous.
Under the above assumptions for each (real-valued)
$F\in C(X\times G)$ such that $F\perp L^2(X,m_X)$, we have
$$\frac1N\sum_{n\leq N}F((R_\varphi)^{pn}(x,g))
F((R_\varphi)^{qn}(x,g))\to0$$
for each $(x,g)\in X\times G$ and different
primes $p,q$ sufficiently large.\end{Lemma}
\begin{proof} First notice that any accumulation point $\rho$ of $\left(\frac1N\sum_{n\leq N}\delta_{(R^p\times R^q)^n(x,x)}\right)$ is ergodic (cf.\ e.g.\ \cite{Ku-Le}), hence  is graphic. It follows that
any accumulation point $\widetilde\rho$ of $\left(\frac1N\sum_{n\leq N}\delta_{((R_\varphi)^p\times(R_\varphi)^q)^n((x,g),(x,g))}\right)$ will be the relatively independent extension of an isomorphism between $R^p$ and $R^q$ in view of the second part of Lemma~\ref{l3}. By the definition of the relatively independent
extension, $\int F\ot F\,d\widehat\rho=0$ and the result follows.\end{proof}

\subsection{Application - KBSZ criterion}\label{s:kbszcrit}
Recall the following result (to which we refer as the KBSZ criterion)
about the orthogonality of numerical sequences with
bounded, multiplicative functions:

\begin{Th}[K\'atai \cite{Ka},
Bourgain, Sarnak and Ziegler \cite{Bo-Sa-Zi}]\label{t:kbsz}
Let a bounded sequence $(a_n)\subset\C$ satisfy
$$
\lim_{N\to\infty}\frac1N\sum_{n\leq N}a_{pn}\ov{a}_{qn}=0$$
for each $p\neq q$ sufficiently large prime numbers. Then
$$
\lim_{N\to\infty}\frac1N\sum_{n\leq N}a_n\bfu(n)=0$$
for each bounded, multiplicative function $\bfu:\N\to\C$.
\end{Th}

In the context of topological dynamical systems, that is, given $(X,S)$,
we use this result with $a_n=f(S^nx)$ with $f\in C(X)$ and $x\in X$.
It is not hard to see that this criterion applies
for any uniquely ergodic $(X,S)$ with the property that
$S^p\perp S^q$ (disjointness is meant if we consider the unique
invariant measure $\mu$) and we consider $f\in C(X)$ with
$\int_Xf\,d\mu=0$ and arbitrary $x\in X$. However, even if we do
not have disjointness of sufficiently large (prime) powers, we
can apply this criterion for {\bf particular} continuous functions
if we control the limit joinings, as, for example, in Lemma~\ref{l4}.

\section{Basics on substitutions of constant length} \label{s:subst}
Let $\A$ be an alphabet (a non-empty finite set). Denote
$\A^\ast:=\bigcup_{m\geq 0}\A^m$, where $\A^m$ stands
for the set of words $w=a_0a_1\ldots a_{m-1}$ over $\A$ of
length $|w|$ equal to~$m$ ($\A^0$ consists only of the empty word). Fix $\N\ni\la\geq2$.
By a {\em substitution of (constant) length} $\la$ we mean a map
$$
\theta:\A\to \A^\lambda$$
which we also write as $\theta(a)=\theta(a)_0\theta(a)_1\ldots\theta(a)_{\la-1}$ for $a\in\A$.
Via the concatenation of words, there is a natural extension of $\theta$ to a map from $\A^m$ to $\A^{m\lambda}$ (for each $m\geq1$) or from $\A^\ast$ to itself,
or even from $\A^{\Z}$ to itself. In particular, we can iterate
$\theta$ $k$ times:
$$
\A\stackrel{\theta}{\to} \A^{\lambda}\stackrel{\theta}{\to}\ldots
\stackrel{\theta}{\to}\A^{\lambda^k}$$
which can be viewed as the substitution $\theta^k$
(the $k$th-iterate of $\theta$) of length $\lambda^k$:
$$
\theta^k:\A\to \A^{\lambda^k}.$$

The following formula is well-known and follows directly by definition:
\beq\label{wzorkoc}
\theta^{k+\ell}(a)_{j'\la^k+j}=\theta^k(\theta^{\ell}(a)_{j'})_j\eeq
for each $a\in\A$, $j'<\la^\ell$, $j<\la^k$.
Indeed, $|\theta^{\ell}(a)|=\la^\ell$ and consider
$\theta^{\ell}(a)_{j'}$, i.e. the $j'$-th letter in the word
$\theta^{\ell}(a)$. We now let act $\theta^k$ on
$\theta^{\ell}(a)$  which transforms letters in the word $\theta^{\ell}(a)$
into blocks of length~$\la^k$, in particular the $j'$-th letter of
$\theta^\ell(a)$ becomes the $j'$-th block of length $\la^k$. So
counting $j$-th letter in this block is the same as counting
$(j'\la^k+j)$-th letter in $\theta^{k+\ell}(a)$.

The {\em subshift} $X_\theta\subset A^{\Z}$ is determined by all
words that appear in $\theta^k(a)$ for some $k\geq1$ and $a\in\A$:
\beq\label{su0}
X_\theta:=\begin{array}{c}\left\{x\in\A^{\Z}:\:
\mbox{for each $r<s$, we have}\right.\\ x[r,s]=\theta^{k}(a)[i,i+s-r-1]\\
\mbox{for some $k\geq1$, $a\in\A$ and $0\leq i<\la^k-(s-r)$}\}.
\end{array}
\eeq
That is, $X_\theta$ is closed and invariant for the left shift $S$
acting on $\A^{\Z}$.
Then, clearly, we have
\beq\label{su1}
X_{\theta^k}\subset X_{\theta}\text{ for each }k\geq1.\eeq
Note also that for each $a\in\A$, there exist $k,\ell\geq1$ such that
$$
\theta^k(a)_0=\theta^{k+\ell}(a)_0$$
from which we deduce that there are a letter $a\in\A$ and
$\ell\geq1$ such that
\beq\label{su2}
\theta^\ell(a)_0=a_0.\eeq
Hence, by iterating the substitution $\theta^\ell$, we obtain
a fixed point $u\in\A^{\N}$ for the map $\theta^\ell:
\A^{\N}\to \A^{\N}$. This, similarly to the RHS of~\eqref{su0}
defines a subshift $X_u\subset\A^{\Z}$ for which we have
$X_u\subset X_\theta$, cf.\ Remark~\ref{r:podszifty}.
In general, we do not have equalities in~\eqref{su1}, and
(the more) $X_u$ is a proper subshift of $X_{\theta^\ell}$ above.
The situation changes if we assume that $\theta$ is {\em primitive},
that is, when there exists $k\geq1$ such that
\beq\label{su3}
\mbox{for each $a\in\A$ the word $\theta^k(a)$ contains all
letters from $\A$.}
\eeq
Then, we have equalities in \eqref{su1}, and moreover,
for each $\ell\geq1$ if $u\in\A^{\N}$ satisfies
$\theta^{\ell}(u)=u$ then
$X_u=X_\theta$. Moreover, we have
\begin{Prop}\label{p:minimality}
Assume that $\theta$ is a primitive substitution of constant
length $\la$. Then $(X_\theta,S)$ is minimal. Moreover, there
exists exactly one invariant measure $\mu_\theta$ on $X_\theta$.
\end{Prop}

In what follows, we assume that $\theta:\A\to\A^{\la}$ is primitive. Under this assumption,
it follows directly by the Perron-Frobenius theorem that,  for any letter $a \in \mathbb{A}$, the density
\begin{align*}
  \delta_a := \lim_{n\to \infty} \frac{\#\{j<\la^n: \theta^n(a')_j = a\}}{\la^n}
\end{align*}
exists and is independent of $a' \in \mathbb{A}$. More precisely, first,  denote by $M(\theta) \in \Z^{|\mathbb{A}| \times |\mathbb{A}|}$ the incidence matrix of $\theta$, i.e. $M_{a,a'}(\theta) := |\theta(a')|_a := $ the number of occurrences of $a$ in $\theta(a')$.
The vector $(\delta_a)_{a\in \mathbb{A}}$ is then the unique right eigenvector for the (maximal) eigenvalue $\la$, i.e.
\beq\label{wzor100}
  \la \cdot \delta_{a} = \sum_{a'\in \mathbb{A}} M(\theta)_{a,a'} \cdot \delta_{a'}
\eeq
holds for all $a \in \mathbb{A}$ and uniquely defines $(\delta_a)_{a\in \mathbb{A}}$, {\cite{Qu}}.

As $\theta$ is primitive, we can also assume that for some $a_0\in\A$, we have
\beq\label{su4}
\theta(a_0)_0=a_0.\eeq
By iterating $\theta$ at $a_0$ we obtain $u\in\A^{\N}$ such that $\theta(u)=u$ (and $X_u=X_\theta$).

We now recall the definition of the {\em height} $h(\theta)$ of $\theta$ following \cite{Qu}. Therefore, for $k\geq0$, set
\beq\label{defH0}
  S_k=S_k(\theta) := \{ r \geq 1 : u[k+r] = u[k]\}
\eeq
and
\begin{align}\label{wyso}
  g_k := \gcd S_k.
\end{align}
This allows us to define
\begin{align}
  h(\theta) := \max\{m \geq 1 :\: (m,\lambda) = 1, m | g_0\}.
\end{align}
We list now some basic properties of $h=h(\theta)$.
\begin{enumerate}
  \item $1 \leq h \leq |\A|$.
  \item $h = \max \{m \geq 1 :\: (m,\lambda) = 1, m | g_k\}$ for every $k\geq 0$.
  \item If, for $j = 0,\ldots, h-1$, we consider the set
    \begin{align*}
      C_j := \{u[j+rh]:\: r\geq0\},
    \end{align*}
  then, by 2., the letter $u[j]$ can be equal to some $u[s]$ only if $s\in C_j$.
  Hence, the sets $C_j$ form a partition of $\A$. If we identify, in $u$, the letters in the same set $C_j$, $j=0,\ldots,h-1$, we thus obtain a periodic (of period~$h$) sequence (cf.~\eqref{susp14}), and $h$ is the largest integer $\leq |\A|$, coprime to $\la$, with this property.
\end{enumerate}

Moreover, if by $\A^{(h)}$ we denote the set $\{u[mh(\theta),mh(\theta)+h(\theta)-1]:\:m\geq0\} \subseteq \A^h$ (which is of course finite) then we can define $\eta:\A^{(h)}\to(\A^{(h)})^{\la}$ by setting
$$
\eta(w)=\theta(w)[0,h(\theta)-1]\theta(w)[h(\theta),2h(\theta)-1]
\ldots \theta(w)[(\la-1)h(\theta),\la h(\theta)-1].$$
Then $\eta$ turns out to be a primitive substitution and $h(\eta) = 1$. Moreover, in view of 3.\ above, as a (measure-theoretic) dynamical system, $(X_\theta,\mu_\theta,S)$ is isomorphic to the $h(\theta)$-discrete suspension $\widetilde{S}$ of $(X_\eta,\mu_\eta,S)$. In what follows we will denote $\eta$ by $\theta^{(h)}$.

If $\theta$ is primitive than so is $\theta^k$ (for each $k\geq1$). Take now $u$ a fixed point for $\theta$. This  is also a fixed point of $\theta^k$. Moreover, $h=h(\theta)=h(\theta^k)$. It follows easily that
\beq\label{hei1} (\theta^{(h)})^k=(\theta^k)^{(h)}.\eeq
Note also that if $(a_0,\ldots, a_{h-1})\in \A^{(h)}$ then
$$
\theta^{(h)}(a_0,\ldots,a_{h-1})=\theta(a_0)\theta(a_1)\ldots\theta(a_{h-1}),$$
where the block on the RHS is dived into blocks of length $h$  which are elements of $\A^{(h)}$. The $j$-th such element ($j<\la$) has the beginning at the position $jh$ which is of the form
$$
jh=i\lambda+y,\; 0\leq y<\lambda$$
with $i=[jh/\lambda]$ (and $y=jh-i\lambda$). It follows that if
$$
\theta^{(h)}(a_0,\ldots,a_{h-1})_j=(b_0,\ldots,b_{h-1})$$
then $b_0=\theta(a_i)_{jh-i\la}$. Replacing in this reasoning $\theta$ by $\theta^k$ and using~\eqref{hei1}, we obtain the following: If $(a_0,\ldots, a_{h-1})\in A^{(h)}$, $k\geq1$, $j<\la^k$ and
$$
(\theta^{(h)})^k(a_0,\ldots,a_{h-1})_j=(b_0,\ldots,b_{h-1})$$
then
\beq\label{hei2}
b_0=\theta^k(a_i)_{jh-i\la^k},\text{ where } i=[jh/\la^k].\eeq

Following \cite{De}, a substitution $\theta$ is called {\em pure} if $\theta=\theta^{(h)}$, i.e.\ if $h(\theta)=1$.

Denote by $c=c(\theta)$ the {\em column number} of $\theta$.
Recall its definition: we consider iterations
$\theta^k:\A\to \A^{\lambda^k}$ and each time we consider sets
$\{\theta^k(a)_j:\:a\in \A\}$ for $j=0,\ldots,\la^k-1$,
where $\theta^k(a)=\theta^k(a)_0\theta^k(a)_1\ldots\theta^k(a)_{\la^k-1}$.
Then $c(\theta)$ is defined as
the  minimal cardinality (we run over $k\geq1$ and $0\leq j\leq\la^k-1$) of such sets.

In what follows we will make use of the following observation.
    \begin{Lemma}\label{l:hc}
      Let $\theta: \A \to \A^{\la}$ be a primitive substitution with
      $c(\theta) = 1$, then $\theta$ is pure, i.e.\ we have $h(\theta) = 1$.
    \end{Lemma}
    \begin{proof}
By assumption, there exist $a\in\A$, $k \in \mathbb{N}$ and $0 \leq j < \lambda^{k}$ such that
      $\theta^k(x)_j = a$ for all $x \in \A$.
Recalling that the fixpoint of $\theta$ is $u$
      we find that $u[j] = u[j + \la^k] = a$. This shows that
      $\la^k \in S_j(\theta)$ and thus, $g_j\mid \la^k $. As $h(\theta)|g_j$ and is coprime with $\lambda$, we have
      $h(\theta) = 1$.
    \end{proof}

In fact, more is true:

\begin{Lemma}\label{le:joining_height0}
Let $\theta:\A\to\A^{\la} $ be a primitive substitution. Then $$c(\theta)=h(\theta)\cdot c(\theta^{(h)}).$$ In particular,  $h(\theta)\mid c(\theta)$.\end{Lemma}

This result seems to be a folklore but we could not find a proof of it in the literature. We postpone the proof of Lemma~\ref{le:joining_height0} to Section~\ref{s:chpbrevisited}.

The following result is well known for pure substitutions but a use of Lemma~\ref{le:joining_height0} yields the following:

\begin{Prop}\label{p:cextension}
Assume that $\theta:\A\to\A^{\la}$ is primitive. Then, as the measure-theoretic
dynamical system, $(X_\theta,\mu_\theta,S)$ is isomorphic to
a $c(\theta)$-point extension (cf.\ Remark~\ref{r:efi}) of the $\la$-odometer
$(H_{\la},m_{H_{\la}},R)$.\end{Prop}
\begin{proof} We have an isomorphism of $X_\theta$ (we omit automorphisms and measures) with $\widetilde{X_{\theta^{(h)}}}$. If $\pi:X_{\theta^{(h)}}\to H_\lambda$ denotes the factor map which is $c(\theta^{(h)})$ to~1 (a.e.), then $\pi\times Id_{\Z/h\Z}$ yields a factor map between the $h$-discrete suspensions $\widetilde{X_{\theta^{(h)}}}$ and $\widetilde{H_\lambda}$. It is again $c(\theta^{(h)})$ to~1 (a.e.). However, in view of~\eqref{suspodom}, the suspension $\widetilde{H_\lambda}$ is isomorphic to the direct product $H_\lambda\times\Z/h\Z$.\footnote{The factor $H_\lambda$ is represented in $H_\lambda\times \Z/h\Z$ as the first coordinate $\sigma$-algebra and it is a factor of $\widetilde{H_\lambda}$ represented by an invariant $\sigma$-algebra. However, because of ergodicity and the fact that $H_\lambda$ has discrete spectrum, there is only one invariant $\sigma$-algebra in $\widetilde{H_\lambda}$ representing $H_\la$. The same argument shows that $H_\lambda$ is a factor of $X_\theta$ in a canonical way.} This makes $H_\lambda$ a factor of $X_\theta$ with (a.e.) fiber of cardinality $c(\theta^{(h)})\cdot h$. The result follows from Lemma~\ref{le:joining_height0}.
\end{proof}

As a matter of fact, whenever $\theta$ is primitive and $h(\theta)=1$, the system $(H_\lambda,R)$ represents the so called maximal equicontinuous factor of $(X_\theta,S)$. The
corresponding factor map is usually seen in the following way:  Each point $x\in X_{\theta}$ has a unique $\la^t$-skeleton structure. By that, one means a sequence $(j_t)_{t\geq1}$ with $0\leq j_t<\la^t-1$ for which, for each $t\geq1$, we have
$$
x[-j_t+s\la^t,-j_t+(s+1)\la^t-1]=\theta^t(c_s)$$
for each $s\in\Z$ and some $c_s\in\A$. Now, the map $x\mapsto(j_t)$ yields the factor map which we seek.

Notice also that if $\mathcal{R}\subset \A\times\A$ is an equivalence relation which is $\theta$-{\em consistent}, that is:
\beq\label{su5}
(a,b)\in\mathcal{R}\;\Rightarrow (\theta(a)_j,\theta(b)_j)\in \mathcal{R}\text{ for each }j=0,\ldots,\la-1,\eeq
then we have the quotient substitution $\theta_{\mathcal{R}}:\A/\mathcal{R}\to\left(
\A/\mathcal{R}\right)^{\la}$ given by
$$
\theta_{\mathcal{R}}([a]):=[\theta(a)_0]\ldots[\theta(a)_{\la-1}]$$
is correctly defined and the dynamical system $(X_{\theta_{\mathcal{R}}},S)$ is a (topological) factor of $(X_{\theta},S)$. Let alone $(X_{\theta_{\mathcal{R}}}, \mu_{\theta_{\mathcal{R}}},S)$ is a measure-theoretic factor of $(X_\theta,\mu_\theta,S)$   (clearly, $\theta_{\mathcal{R}}$ is primitive). A particular instance of a $\theta$-consistent equivalence relation $\mathcal{R}$ given by
$$
(a,b)\in \mathcal{R}\text{ if and only if } \theta(a)=\theta(b).$$
In this situation, the quotient dynamical system
$(X_{\theta_{\mathcal{R}}},S)$ is in fact (topologically)
isomorphic to $(X_\theta,S)$. Therefore, no harm arises if we assume
that
\beq\label{su6}
\mbox{$\theta$ is 1-1 on letters, i.e., $\theta(a)\neq\theta(b)$
whenever $a\neq b$.}\eeq

Finally, we assume that
\beq\label{su7}
\theta\text{ is {\em aperiodic}},
\eeq
that is, there is a non-periodic element $x\in X_\theta$.
In fact, since we assume that $\theta$ is primitive,
$\theta$ is aperiodic if and only if $X_\theta$ is infinite.

From now on, we consider only  substitutions $\theta$ satisfying \eqref{su3}, \eqref{su4},~\eqref{su6} and~\eqref{su7}. Often, we will additionally assume that $\theta$ is pure.

\section{Proof of Theorem~\ref{t:main}~(i)}
This and Section~\ref{s:bij} will be devoted to prove
Theorem~\ref{t:main} in some particular cases.
We assume that $c(\theta)=1$, so according to
Proposition~\ref{p:cextension} we deal (from ergodic theory point of
view) with discrete spectrum case.\footnote{A prominent example in
this class is the Baum-Sweet sequence given by the substitution
$a\mapsto ab$, $b\mapsto cb$, $c\mapsto bd$ and $d\mapsto dd$. It
is not hard to see $\{\theta^3(y)_5:\;y\in\{a,b,c,d\}\}=\{d\}$.}
In fact, we will prove even a stronger property.

So we have $\theta:\A\to\A^{\la}$, and we assume that for some $a\in\A$, $\theta(a)_0=a$. Moreover, by replacing $\theta$ by its iterate if necessary, we can assume that
$$\left|\{0\leq j<\la:\:|\{\theta(b)_j:\:b\in\A\}|=1\}\right|:=\ell_1\geq1.$$
Hence, since $u=\theta(u)$ is a concatenation of words $\theta(b)$, $b\in\A$, in the interval $[0,\la-1]$ the number of positions $j$, for which $u[j+s\la]=u[j]$ for each $s\geq1$ is $\ell_1$. Let us pass to $\theta^2$. We are interested in
$$
\ell_2:=\left|\{0\leq j<\la^2-1:|\{\theta^2(b)_j:\:b\in\A\}|=1\}\right|.$$
We have $\ell_2\geq(\la-\ell_1)\ell_1+\ell_1\la$. Inductively, if
$$
\ell_{k-1}:=\left|\{0\leq j<\la^{k-1}-1:|\{\theta^{k-1}(b)_j:\:b\in\A\}|=1\}\right|$$
then
\beq\label{lk}
\ell_k\geq(\la-\ell_1)\ell_{k-1}+\ell_1\la^{k-1}.\eeq
Moreover, by definition,
\beq\label{lk1}
\mbox{the sequence $\ell_k/\la^k$, $k\geq1$, is increasing.}\eeq
In view of~\eqref{lk}, we have
$$
\frac{\ell_k}{\la^k}\geq\frac{\la-\ell_1}{\la}\frac{\ell_{k-1}}
{\la^{k-1}}+\frac{\ell_1}{\la},$$
and since, by~\eqref{lk1}, the sequence $\ell_k/\la^k$ is convergent, the above recurrence formula implies
\beq\label{lk2}
\lim_{k\to\infty} \frac{\ell_k}{\la^k}=1.
\eeq
Note that an interpretation of $\ell_k$ is that it is the number of coordinates $j$ for $j=0,1,\ldots, \la^k-1$ such that
$u[j+s\la^k]=u[j]$ for each $s\geq1$.

Recall now the notion of Weyl rationally almost periodic sequences (WRAP). A sequence $(x[n])\in \A^{\N}$ is WRAP if it is the limit of periodic sequences in the Weyl pseudo-metric $d_W$:
$$
d_W(z,z')=\limsup_{N\to\infty}\,\sup_{m\geq0}
\frac1N\left|\{m\leq n\leq m+N-1:z[n]\neq z'[n]\}\right|.$$
Now, \eqref{lk2} implies immediately that:
\begin{Prop}\label{p:cn=1}
If $c(\theta)=1$ then the fixed point $u=\theta(u)$ is WRAP, where the periodic sequences can be chosen with period $\la^k$.\end{Prop}

From the point of view of M\"obius disjointness, dynamical systems $(X_x,S)$ given by WRAP sequences $x$ have already been
studied in \cite{Be-Ku-Le-Ri} and
it is proved there that all continuous observables
$(f(S^ny))$ (for $f\in C(X_x)$) are orthogonal to the M\"obius function $\mob$.
But a rapid look at the proof in \cite{Be-Ku-Le-Ri}  shows that the only property of
$\mob$ used in it was the aperiodicity of $\mob$
(what is essential in the proof is that all points $y\in X_x$
are also WRAP).

\begin{Cor}[\cite{Be-Ku-Le-Ri}]\label{c:weyl}
If $x\in\A^{\N}$ is WRAP then for each bounded $\bfu:\N\to\C$ which is aperiodic, we have
$$
\lim_{N\to\infty}\frac1N\sum_{n\leq N}f(S^ny)\bfu(n)=0$$
for all $f\in C(X_x)$ and $y\in X_x$.
In particular, the above assertion holds for substitutional dynamical systems with $c(\theta)=1$.~\footnote{The automatic sequences given by observables in case $c(\theta)=1$ represent so called {\em synchronized} case in \cite{Mu}.}\end{Cor}

\begin{Remark} Theorem~\ref{t:main}~(i) also follows
from~\cite{De-Dr-Mu} but because the relations between
synchronized automata and substitutions with $c(\theta)=1$
do not seem to be explained explicitly in literature,
we gave a more general and direct argument.\end{Remark}

\section{More ergodic theory prerequisites}
\subsection{Essential centralizer}\label{s:essencent} Assume that $T$ is an ergodic\footnote{We assume also that $T$
is {\em aperiodic}: for $\nu$-a.e.\ $y\in Y$,
the map $m\mapsto T^my$ is one-to-one on $\Z$.} automorphism of $\ycn$. By the {\em centralizer}  $C(T)$ of $T$ we mean the group of all invertible
automorphisms $V\in{\rm Aut}\,\ycn$ commuting with $T$.
Clearly, $\{T^n:\:n\in\Z\}$ is a normal subgroup of $C(T)$ and the group $EC(T) := C(T) / \{T^n: \:n\in \Z\}$ is called the {\em essential centralizer} of $T$.

\begin{Lemma}\label{l:essc} Assume that $EC(T)$ is finite and $C(T)=C(T^p)$ for all sufficiently large $p\in\mathscr{P}$. Then, for all sufficiently large $p,q\in\mathscr{P}$, $p\neq q$, the automorphisms  $T^p$ and $T^q$ are not isomorphic.
\end{Lemma}
\begin{proof} Since $EC(T)$ is finite, we have
$$
C(T)=\{T^nV_i:\: n\in\Z, i=0,1,\ldots,K\},$$
where $V_0=Id$ and $V_i\in C(T)\setminus\{T^j:\:j\in\Z\}$ for $i=1,\ldots, K$. For $i=1,\ldots,K$, let $m_i\geq1$ be the smallest natural number such that $V_i^{m_i}\in\{T^j:\:j\in\Z\}$ (of course $m_i$ is precisely the order of the coset given by $V_i$ in $C(T)/\{T^j:\:j\in\Z\}$). In what follows, we consider only prime numbers $p,q$ which are dividing no $m_i$ for $i=1,\ldots,K$. Suppose that for $p\neq q$ (sufficiently large) we have an isomorphism of $T^p$ and $T^q$. As obviously $T^q$ has a $q$-root, it follows that there is a root of degree $q$ of $T^p$, i.e.\ there exists $W\in{\rm Aut}\,\ycn$ such that $W^q=T^p$.\footnote{Note that if $\tau$ is an automorphism and $\tau=\sigma^q$ for another automorphism $\sigma$, then $\sigma\in C(\tau)$ as
$\sigma\circ \tau=\sigma\circ\sigma^q=\sigma^q\circ \sigma=\tau\circ\sigma$.} Now, $W\in C(T^p)$, hence (by assumption) $W\in C(T)$. It follows that, for some $n\in\Z$ and $0\leq i\leq K$, $W=T^n\circ V_i$, whence
$W^q=(T^n\circ V_i)^q$ and we obtain
$$V_i^q=T^{p-nq}$$
which, if $i\neq0$, is impossible as $q$ is coprime to  $m_i$ and if $i=0$, $T^p=T^{nq}$ which yields $q|p$ (or $p=0$ if $n=0$), a contradiction.
\end{proof}
\begin{Remark}\label{r:essc} Notice that in the above proof,
in fact we proved that $T^p$ cannot be isomorphic to $U^q$
with $U\in{\rm Aut}\,\ycn$, in other words, we proved
that $T^p$ cannot have a $q$-root.\end{Remark}

\subsubsection{Centralizer of $h$-discrete suspensions}
We assume that $T\in{\rm Aut}\ycn$ and let $\widetilde{T}$ denote its $h$-discrete suspension, see~\eqref{susp11} and~\eqref{susp12}. Note that whenever $V\in C(T)$, the formula $\ov{V}(y,j):=(Vy,j)$ for $(y,j)\in\widetilde{Y}$ defines an element of the centralizer of $\widetilde{T}$. In fact, we have the following:

\begin{Prop}[Cor. 1.4 in \cite{Da-Le}] \label{p:centsusp}
If $T\in{\rm Aut}\ycn$ is ergodic then $C(\widetilde{T})=\{\ov{V}\circ \widetilde{T}^m:\:V\in C(T),m\in\Z\}$.\end{Prop}

It follows immediately from Proposition~\ref{p:centsusp} that:

\begin{Cor}\label{c:centsusp}
If $T\in{\rm Aut}\ycn$ is ergodic and $EC(T)$ is finite, also
$EC(\widetilde{T})$ is finite.\end{Cor}

\subsubsection{Essential centralizer of substitutions of constant length}
The result below has been proved by Host and Parreau in \cite{Ho-Pa} for pure substitutions of constant length. However, taking into account Corollary~\ref{c:centsusp} and  the fact that each substitutional system is an $h$-discrete suspension of its pure basis (which is also given by a substitution of constant length), see Section~\ref{s:subst}, we obtain the
following\footnote{For the particular instance of bijective
substitutions, the result was proved to hold slightly
earlier in \cite{Le-Me}.} result.

\begin{Th}[\cite{Ho-Pa}]\label{t:essc1}
Let $\theta$ be a primitive substitution of constant length. Then (remembering that $(X_\theta,S)$ is uniquely ergodic),  $EC(S)$ is finite if and only if $c(\theta)>h(\theta)$.
\end{Th}

\begin{Remark}\label{r:unc}
If $c(\theta)=h(\theta)$ (in particular, if $c(\theta)=1$) then we are in  the synchronizing case.
Thus, the spectrum of the corresponding dynamical system is discrete.
Therefore, the essential centralizer is uncountable.
In fact, $S^p$ is isomorphic to $S^q$ for all sufficiently large $p,q\in\mathscr{P}$.
\end{Remark}

\begin{Remark}\label{r:bijcent} Assume that $V:\A\to\A$ is a
bijection ``commuting'' with $\theta$:
\beq\label{comm}
\theta(V(a))_j=V(\theta(a)_j)\eeq
for each $j=0,1,\ldots, \la-1$. Then, we claim that
$V$ has a natural extension to
a homeomorphism $V:X_\theta\to X_\theta$ so that $S\circ V=V\circ S$.
Indeed, treat $V$ as a 1-code map on $\A^\Z$ (so obviously, it commutes
with the shift). Clearly, we only need to show that
$$
y\in X_\theta\;\Rightarrow\; Vy\in X_\theta.$$
Now, for each $i\leq j$, $y[i,j]=\theta^k(a)[m,m+(j-i)]$
for some $a\in\A$, $k\geq1$ and $m\geq0$. It follows
from~\eqref{comm} that
$$(Vy)[i,j]=V\theta^k(a)[m,m+(j-i)]=\theta^k(Va)[m,m+(j-i)]
$$
and the claim follows. Note that if $\theta$ is primitive, then
$V:X_\theta\to X_\theta$ preserves $\mu_\theta$, so $V\in C(S)$.
\end{Remark}

\subsection{Joinings of powers of finite extensions of odometers}
\begin{Lemma}[\cite{Fe-Ku-Le-Ma}]\label{l:essc2}
Assume that $T$ acting on $\ycn$ is ergodic (aperiodic) and has discrete spectrum.
Let $\va:Y\to G$ be a cocycle with $G$ finite. Assume that $T_{\va}$ is ergodic.
Moreover, assume that for $p\in\mathscr{P}$ large enough the corresponding group extension $\left(T_{\va}\right)^p$ is also ergodic.\footnote{This assumption requires small spectrum of $U_T$.}  Then for $p\in\mathscr{P}$ large enough, we have
$C(T_{\va})=C\left(\left(T_{\va}\right)^p\right)$.
\end{Lemma}

Using Lemmas~\ref{l:essc} and~\ref{l:essc2}, we obtain the following.

\begin{Prop}\label{p:essc2}
Assume that $T_{\va}$ is an ergodic  $G$-extension ($G$ a finite group) $T\in{\rm Aut}\,\ycn$ with discrete spectrum, so that $EC(T_{\va})$ is finite.
  Assume moreover that $(T_{\va})^p$ is ergodic for all $p\in\mathscr{P}$ sufficiently large. Then $(T_{\va})^p$ and $(T_{\va})^q$ are not isomorphic for all $p,q\in\mathscr{P}$ which are different and large enough.
\end{Prop}

In case of odometers, we can prove more. Assume that $R$ is the odometer given by $\lambda$ (that is, $R\in{\rm Aut}\,(H_\lambda,m_{H_\lambda})$ is  ergodic, has discrete spectrum and the eigenvalues are roots of unity of degree $\la^k$, $k\geq1$).
Consider  $p,q\in\mathscr{P}$ relatively prime  with $\la$.
Then $R$ is isomorphic to $R^p$ (also to $R^q$), so $R^p$ is
isomorphic to $R^q$. Recall also that the only {\bf ergodic} joinings
between $R^p$ and $R^q$
are given by the graphs of isomorphisms between $R^p$
and $R^q$.

\begin{Prop}\label{p:essc3} Assume that $R$ is the $\la$-odometer.
Let $\va:H_\lambda\to G$ be a cocycle with $G$ finite,
so that $R_{\va}$ is ergodic, and $(R_{\va})^p$ is also ergodic
for $p\in\mathscr{P}$ large enough. Assume moreover
that $EC(R_{\va F})$ is finite for each proper subgroup
$F\subset G$.\footnote{\label{f:maintool} We recall that the factor,
as automorphism, is given by $R_{\va F}(x,gF)=(Rx,\va(x)gF)$. Note that, when applied to substitutions, this assumption requires $H_\lambda$ to represent the Kronecker factor, cf.\ Remark~\ref{r:unc}. However, if we replace $H_\lambda$ by $H_\lambda\times\Z/h\Z$, this assumption is satisfied and the proof of Proposition~\ref{p:essc3} remain the same.}
Then, for each different $p,q\in\mathscr{P}$ large enough,
the {\bf only} joinings\footnote{Note that we drop
the assumption of ergodicity, cf.\ the second part of Lemma~\ref{l3}.} between $(R_{\va})^p$
and $(R_{\va})^q$ that project onto  ergodic joinings
of $R^p$ and $R^q$ are relatively independent extensions
of the projections.
\end{Prop}
\begin{proof} A general theory of groups extensions (see e.g.\ \cite{Le-Me1}, \cite{Me}) tells us that if
$$\rho\in J^e((R_{\va})^p, (R_{\va})^q)=J^e(R^p_{\va^{(p)}},R^q_{\va^{(q)}})$$
then there exist $F_1,F_2\subset G$ normal subgroups of $G$ such that an isomorphism $W$ between $R^p$ and $R^q$   (we have assumed that $\rho|_{H_\lambda\times H_\lambda}$ is the graph of $W$) lifts to
an isomorphism $\widetilde{W}$ of the factors given by $H_\lambda\times G/F_1$ and $H_\lambda\times G/F_2$. But $R^p_{\va^{(p)}F_1}=(R_{\va F_1})^p$. Moreover, $F_1,F_2$ depend on $p,q$ but altogether we have only finitely many possibilities for $F_1,F_2$. Assume that $F_1$ is a proper subgroup of $G$.
We use our assumption and Lemma~\ref{l:essc} together with Remark~\ref{r:essc} to get a contradiction. It follows that we can obtain only relatively independent extensions of graphs as joinings between $(R_\varphi)^p$ and $(R_\varphi)^q$ whenever $p\neq q$ are large enough.
\end{proof}

\begin{Remark}\label{r:essc1} By Mentzen's theorem \cite{Me1}
on (partly continuous spectrum) factors of substitutions of constant length, when $R_{\va}$
is given by a substitution (non-synchronizing, primitive)
then the natural factors $R_{\va F}$ for $F$ proper subgroup of
$G$ are also (up to measure-theoretic isomorphism)
substitutions (cf.\ Footnote~\ref{f:maintool}),
so the assumptions of Proposition~\ref{p:essc3} will be satisfied.
\end{Remark}

\subsection{Strategy of the proof of the main result}
\label{s:mainstrategy}
To prove Theorem~\ref{t:main}, we first show that
each substitutional system  $(X_{\theta},S)$ is a
{\bf topological} factor of another
substitutional system $(X_{\widehat\Theta},S)$,
where, from the measure-theoretic point of view,
the system $(X_{\widehat\Theta},S)$, which is uniquely ergodic, is
isomorphic to $R_\va$ satisfying
the assumption of Proposition~\ref{p:essc3},
see Remark~\ref{r:essc1}. Moreover, we will show
that the odometer of the original substitutional system is $R$.

Let $(H_\lambda,R)$ denote the $\la$-odometer associated with
$(X_{\widehat\Theta},S)$. If the height of $\widehat\Theta$ is one, the $\la$-odometer  is the maximal equicontinuous factor\footnote{If $h(\widehat\Theta>1$ then we must replace $H_\lambda$ with $H_\lambda\times \Z/h\Z$.}
of $(X_{\widehat\Theta},S)$. Hence, there is a continuous
equivariant map $\pi:X_{\widehat\Theta}\to H_\lambda$.
Moreover, the following observation we already used in the proof of Lemma~\ref{l4}:

\begin{Lemma}\label{l:yy} Take different $p,q\in\mathscr{P}$ sufficiently large.
Then, for each $y\in H_\lambda$,
the point $(y,y)$ is generic for an ergodic
$R^p\times R^q$ invariant measure $\kappa$, i.e.\ for each $F\in C(H_\lambda\times H_\lambda)$, we have
$$
\frac1N\sum_{n\leq N}F(R^{pn}y,R^{qn}y)\to\int_{H_\lambda\times H_\lambda}F\,d\kappa.$$.\end{Lemma}

It follows that $(y,y)$ is generic for $\kappa=(m_{H_\lambda})_W$, the graph joining
given by an isomorphism $W$ between $R^p$ and $R^q$.

Take any $x\in X_{\widehat\Theta}$. We study now the sequence
($p\neq q$ sufficiently large)
$$
\frac1N\sum_{n\leq N}\delta_{(S^p\times S^q)^n(x,x)},\;N\geq 1.$$
Any of its limit points yields a joining
$\rho\in J((R_{\va})^p,(R_{\va})^q)$.
Now, $\rho|_{H_\lambda\times H_\lambda}$ is {\bf precisely}
obtained as the limit of
$$
\frac1{N_k}\sum_{n\leq N_k}
\delta_{(R^p\times R^q)^n(\pi(x),\pi(x))},$$
for a relevant subsequence $(N_k)$. But we have already
noticed that any such limit must be an {\bf ergodic}
joining of $R^p$ and $R^q$, hence it is a graph.
Now, we use our results to obtain that
the limit of $\frac1{N_k}\sum_{n\leq N_k}
\delta_{(S^p\times S^q)^n(x,x)}$ also exists and it is
the relatively independent extension of the underlying graph
joining between $R^p$ and $R^q$.
We have proved the following.
\begin{Prop}\label{p:essc4} If different $p,q\in\mathscr{P}$
are large enough, then the set
$$
\{\rho\in J(S^p,S^q):\:\rho=\lim_{k\to\infty}
\frac1{N_k}\sum_{n\leq N_k}\delta_{(S^p\times S^q)^n(x,x)
}\text{ for some }N_k\to\infty\}
$$
is contained in the set of relative products over
the graphs of isomorphisms between $R^p$ and $R^q$.\footnote{In fact, by Lemma~\ref{l:yy}, it follows that each point $(x,x)$ is generic.}
\end{Prop}

Assume now that
\beq\label{warunek}
\mbox{$F\in C(X_{\widehat\Theta})$ and
$F\perp L^2(\pi^{-1}(\cb(H_\lambda)))$.}\eeq
Again, fix $x\in X_{\widehat\Theta}$. We have
$$
\frac1{N_k}\sum_{n\leq N_k}F(S^{pn}x)\ov{F(S^{qn}x)}\to
\int_{X_{\widehat\Theta}\times X_{\widehat\Theta}}
F\ot\ov{F}\,d\rho=$$$$
\int_{H_\lambda\times H_\lambda}\E(F\ot\ov{F}|H_\lambda\times H_\lambda)\,d\rho|_{H_\lambda\times H_\lambda}=
\int_{H_\lambda}\E(F|H_\lambda)\cdot\ov{\E(F|H_\lambda)\circ W}\,d\rho=0.$$
by the definition of the relative product.

The only thing which is missing now is to be sure that we have
sufficiently many functions $F$ satisfying~\eqref{warunek}.  In fact,
we aim at showing that $(X_{\widehat\Theta},S)$ has a topological
factor $(X_{\theta'},S)$, which measure-theoretically is equal to
$(H_{\la},R)$ and for each $F\in C(X_{\widehat\Theta})$, we have
\beq\label{dobre}
F=F'+F'',\eeq
where $F'\in C(X_{\theta'})$ (with some abuse of notation), and
$F''\perp L^2(\pi^{-1}(\mathcal{B}(H_\lambda)))$
(of course, $F''$ is also continuous).
This will allow us
to conclude
the proof of Theorem~\ref{t:main} using Theorem~\ref{t:kbsz} for $F''$
and dealing separately with $F'$.

In  the bijective case, the existence of a ``good''
$\widehat{\Theta}$ is known, see Section~\ref{s:bij} below:
in $(X_{\widehat\Theta},S)$ we have many ``good''
continuous function, as the Kronecker factor of
$(X_{\widehat\Theta},S)$ has a ``topological'' realization
(see \cite{Fe-Ku-Le-Ma}). However, in the general case such an
approach seems to be unknown, and in Section~\ref{s:proof},
we will show a new general construction of an extension of
a substitutional system in which we will see sufficiently
many continuous functions satisfying~\eqref{warunek}.

\section{Proof of Theorem~\ref{t:main} in the bijective and
quasi-bijective case}\label{s:bij}
A substitution $\theta:\A\to\A^{\la}$ is called {\em quasi-bijective}\footnote{The Rudin-Shapiro substitution is a prominent example of a quasi-bijective substitution which is not bijective.} if for all $n\geq n_0$ and all $j\in\{0,\ldots,\lambda^n-1\}$, we have
$$
c(\theta)=|\{\theta^n(a)_j:\:a\in \A \}|.$$
If, additionally, $c(\theta)=|\A|$, then we speak about a
{\em bijective}\footnote{Even for a bijective substitution, the height can be $>1$, see Example~3 in \cite{Me1}. Also, substitution $\widehat{\Theta}$ in Example~\ref{ex:bij} enjoys the same properties.} substitution (sometimes,
such a substitution is also called {\em invertible}).
The proof of Theorem~\ref{t:main} in the bijective case
is provided in \cite{Fe-Ku-Le-Ma}. The proof  is also provided
in \cite{Fe-Ku-Le-Ma} for the Rudin-Shapiro substitution and the proof can be repeated for some other quasi-bijective substitutions.
The proof of Theorem~\ref{t:main} covers the general case.

\section{Substitutions of constant length - one more point of view} \label{s:proof}

\subsection{Substitution joinings of substitutional systems}
Assume that we have two substitutions $\theta:\A\to \A^{\la}$
and $\zeta:\B\to \B^{\la}$. Assume that
\beq\label{s1}
\ca\subset \A\times \B,\; p_\A(\ca)=\A,\;p_\B(\ca)=\B,\eeq
where $p_\A$, $p_\B$ stand for the projections on $\A$ and $\B$,
respectively. Moreover, we assume that
\beq\label{s2}
(a,b)\in\ca \;\Rightarrow (\theta(a)_j,\zeta(b)_j)\in\ca
\eeq
for each $j=0,\ldots,\la-1$. Then it is easy to see that the formula
\beq\label{s3}
\Sigma(a,b)=(\theta(a)_0,\zeta(a)_0)(\theta(a)_1,\zeta(a)_1)\ldots
(\theta(a)_{\la-1},\zeta(a)_{\la-1})\eeq
defines a substitution
$\Sigma:\ca\to\ca^{\la}$ of length $\la$. We can also use notation:
$$
\Sigma=\theta\vee\zeta.$$

\begin{Remark}\label{r:sjoi}
We note that in general the above $\Sigma$ need not be primitive, that is, it does not necessarily satisfy \eqref{su3}. Indeed, consider for example $\zeta=\theta$ and then take for $\mathcal{A}$ the product set $\A\times \A$. On the other hand, the ``diagonal'' $\ca:=\{(a,a):\: a\in\A\}$ yields a primitive substitution (clearly isomorphic to $\theta$).

In order to see a less trivial primitive example, consider $\A =\{0,1,2\}$ and $\zeta = \theta$, where (see the last example in Section~2 \cite{De})
$$
0\mapsto 010,\;1\mapsto 102,\; 2\mapsto 201$$
with $\mathcal{A} = \{(0,1), (1,0), (0,2), (2,0)\}$, which gives a primitive substitution different from the ``diagonal'' one.\footnote{It is not hard to see that in this example, $\theta$ is primitive and $c(\theta)=2$. If by $u$ we denote the fixed point of $\theta$ obtained by the iterations of $0$, then $S_0=2\N$, so $g_0=2$ and hence $h(\theta)=2=c(\theta)$. This means that the dynamical system $(X_\theta,\mu_\theta, S)$ has discrete spectrum and hence since $\Sigma$ is also primitive, as measure-theoretic dynamical systems,
$(X_\theta,\mu_\theta,S)$ and $(X_\Sigma,\mu_\Sigma,S)$ are isomorphic, cf.~\eqref{jdis}.} If we slightly change the definition of $\theta$, namely:
$$
0\mapsto 010,\;1\mapsto 201,\;2\mapsto 102$$
(which is still primitive with $c=h=2$) then $\Sigma$ is primitive but~\eqref{su4} is not satisfied.
\end{Remark}

\begin{Remark}\label{r:sjoi11}
Note that $X_{\Sigma}$
is a topological joining of $X_\theta$ and $X_\zeta$. Indeed,
up to a
natural rearrangement
of coordinates, $X_{\theta\vee\zeta}\subset X_\theta\times X_\zeta$.
Then, for every $(x,y)\in X_{\Sigma}$, the orbits of $x$ and $y$
are dense in $X_\theta$ and $X_\zeta$, respectively.\footnote{This
follows by the minimality of
$(X_\theta,S)$ and $(X_\zeta,S)$, respectively.}
Now, the image of  the natural projection $(x,y)\mapsto x$ is
contained in $X_\theta$ and if $\theta(u)=u$, $\zeta(v)=v$ then
$(\theta\vee\zeta)(u,v)=(u,v)$ (after a rearrangement of coordinates).
So $u$ is in the image of $p_{X_\theta}(x,y)=x$ and therefore
\beq\label{factor0}
p_{X_\theta}:X_{\theta\vee\zeta}\to X_\theta, \;
p_{X_\zeta}(X_{\theta\vee\zeta})=X_\zeta.\eeq
Since $p_{X_\theta}$ is continuous and equivariant, it settles a topological
factor map between the relevant substitutional systems.\end{Remark}

\begin{Def} We call $\Sigma$ a {\em substitution joining} of
$\theta$ and $\zeta$ if $\Sigma$ satisfies \eqref{su3} and
\eqref{su4} (note that~\eqref{su6} and~\eqref{su7} are satisfied automatically as $\theta$, $\zeta$ are substitutions).\end{Def}

\subsection{Joining with the synchronizing part}
In general, when dealing with the dynamical system given
by a substitution of constant length, we would like to see its
Kronecker factor as a topological factor realized
``in the same category of objects'', that is, realized by another substitution.
This is not always possible, even in the class of
bijective substitutions, see Herning's example \cite{He}.
For the purpose of orthogonality with an arithmetic function $\bfu$, we need
however only an  extension of the original substitution
which is given by another substitution (of the same length) and
require that in the extended system we have a ``good''
realization of the Kronecker factor.
This is done by a joining of $\theta$ with its synchronizing
part.

\subsubsection[Synchronizing part]{Synchronizing part of $\theta$}
Given a substitution $\theta:\A\to \A^{\la}$, set
$$
\mathcal{X}(\theta) = \mathcal{X} :=\{M\subset \A:\: M\text{ realizes the column number of }\theta\},$$
i.e.
there exist $k_M\geq1$, $j_M\in\{0,\ldots,\la^{k_M}-1\}$ such that
\begin{align*}
  M &=\{\theta^{k_M}(a)_{j_M}:\:a\in \A\},\;|M|=c(\theta).
\end{align*}
We can assume without loss of generality that $k_M = k_{M'}$ for all $M,M'\in \mathcal{X}$.
Note that when $M\in\mathcal{X}$ then for each $j=0,\ldots,\la-1$, we have
$$
\left|\theta(M)_j\right|=c(\theta),$$
where by $\theta(M)_j$ we denote the set $\{\theta(a)_j:\:a\in M\}$.
It follows that the formula
\beq\label{s4}
\widetilde{\theta}(M)=\theta(M)_0\theta(M)_1\ldots
\theta(M)_{\la-1}\eeq
defines a substitution $\widetilde{\theta}:\mathcal{X}\to
\mathcal{X}^\la$.

\begin{Prop}\label{p:listas}
Substitution $\widetilde\theta$ has the following basic properties:\\
(i) $c(\widetilde{\theta})=1$.\\
(ii) $\widetilde\theta$ is primitive.\\
(iii)  $h(\widetilde\theta)=1$.
\end{Prop}
\begin{proof}
For (i),
let $M\in\mathcal{X}$ and suppose that $\theta^{k_M}(\A)_{j_M}=M$.
Then, for each $M'\in\mathcal{X}$,
we have $\theta^{k_M}(M')_{j_M}=M$.
The validity of (ii) follows by the same argument.
Finally, (iii) follows from
Lemma~\ref{l:hc}.\end{proof}

\begin{Def} We call $\widetilde{\theta}$ the {\em
synchronizing part} of $\theta$.\end{Def}

Also, note that
\beq\label{rzut}
\bigcup_{M\in\mathcal{X}}M=\A.\eeq
Indeed, fix $M\in\mathcal{X}$ and let $a\in\A$. Take any $x\in M$ and (by primitivity) choose $k\geq1$ so that $\theta^k(x)_j=a$ for some $0\leq j<\la^k$. Then $a\in\widetilde{\theta}^k(M)_j$ and $\widetilde\theta^k(M)_j\in\mathcal{X}$.

\begin{Remark}\label{r:Mpartition}
If the union in~\eqref{rzut} is additionally a partition, then we
obtain an equivalence relation on $\A$ which is $\theta$-consistent (cf.~\eqref{su5}) and the dynamical system
$(X_{\widetilde\theta},S)$ is a topological factor
of $(X_\theta,S)$.
However, in general, there is no reason for
$(X_{\widetilde\theta},S)$ to be a topological factor of
$(X_\theta,S)$.\end{Remark}

Note that, in general, the union in
\eqref{rzut} is not a partition as the following example shows.

\begin{Example}\label{ex:ex1}
Consider the following substitution $\theta$:
\begin{align}\label{eq:example}
\begin{split}
  \theta(a) &= ab\\
  \theta(b) &= ca\\
  \theta(c) &= ba.
\end{split}
\end{align}
By looking at $\theta^3$, we see that $\theta$ is primitive.
Moreover, straightforward computations give
that $c(\theta) = 2$ and $h(\theta) = 1$. In fact,
$\mathcal{X} = \{\{a,b\},\{a,c\}\}$.
Therefore, the union in \eqref{rzut} is not a
partition.

Note that the first column of~\eqref{eq:example} and all its
iterates yields
set $\{a,b,c\}$, so $\theta$ is not quasi-bijective.

Furthermore by replacing $\{a,b\}$ by 0 and $\{a,c\}$ by 1,
we compute $\widetilde{\theta}$:
\begin{align*}
  \widetilde{\theta}(0) &= 10\\
  \widetilde{\theta}(1) &= 00.
\end{align*}

\end{Example}
We also have the following result:

\begin{Prop}\label{p:charbij+qb}
Let $\theta:\A\to\A^{\lambda}$ be a substitution. Then:
\beq\label{bijtri}
\mbox{$\theta$ is bijective if and only if $\widetilde{\theta}$ is trivial.}\eeq
\beq\label{qbijtri}\begin{array}{c}
\mbox{$\theta$ is quasi-bijective if and only if
$X_{\widetilde{\theta}}$ is finite, i.e.}\\
\mbox{a fixed point $\widetilde{u}$ of $\widetilde\theta$ is periodic}.\end{array}\eeq
\end{Prop}
\begin{proof}
First of all,~\eqref{bijtri} follows directly from \eqref{rzut}.
Let us pass to the proof of~\eqref{qbijtri}.

$\Rightarrow$: We assume that $\theta$ is quasi-bijective, i.e. for all $n\geq n_0$ and all $j<\la^n$, we have
\begin{align*}
  |\theta^n(\A)_j| = c(\theta).
\end{align*}
We define $M_j := \theta^{n_0}(\A)_j$ and it follows immediately that $\widetilde{\theta}^{n_0}(M)_j = M_j$
for all $M \in \mathcal{X}$. Let $\widetilde{u}$ be a fixed point of $\widetilde{\theta}$. Now, for $i\geq 0$ and $0\leq j < \la^{n_0}$, we find that (use~\eqref{wzorkoc} letting $\ell\to\infty$),
\begin{align*}
      \widetilde{u}[j + i \la^{n_0}] = \widetilde{\theta}^{n_0}(\widetilde{u})[j+i \la^{n_0}]
      = \widetilde{\theta}^{n_0}(\widetilde{u}[i])_j = M_j,
    \end{align*}
which shows that $\widetilde{u}$ is periodic (in fact, $\la^{n_0}$ is a period of $\widetilde u$).

$\Leftarrow$: Let $\widetilde{u}$ be a periodic fixed point of $\widetilde{\theta}$ with period $p$. First, we claim  that $\widetilde{u}$ is also periodic with period $\la^{n_0}$ for some $n_0\geq 0$. Indeed,
by Proposition~\ref{p:cn=1} (and Proposition~\ref{p:listas}), it follows that, in the Weyl pseudo-metric $d_W$, we can approximate $\widetilde{u}$ by periodic sequences that have period $\la^n$.
Thus, there exists a $\la^{n_0}$-periodic sequence $v$ such that
\begin{align*}
  d_W(\widetilde{u}, v) < \frac{1}{2p}.
\end{align*}
By basic properties of $d_W$, we obtain
\begin{align*}
  d_W(S^{\la^{n_0}}(\widetilde{u}), S^{\la^{n_0}}(v)) \leq d_W(\widetilde{u}, v) < \frac{1}{2p}.
\end{align*}
Hence, the triangle inequality implies
\begin{align*}
  d_W(\widetilde{u}, S^{\la^{n_0}}(\widetilde{u})) < \frac{1}{p}.
\end{align*}
Since $\widetilde{u}$ and $S^{\la^{n_0}}(\widetilde{u})$ are both $p$-periodic functions with distance less that $1/p$, they must coincide and the claim follows.

As, by Proposition~\ref{p:listas}, $\widetilde{\theta}$ is primitive, there exists $n\geq 0$ such that for each $M \in \mathcal{X}$, we can find $j_1<\la^n$ such that $\widetilde{\theta}^{n}(\widetilde{u}[0])_{j_1} = M$. Moreover, for all $j_0 < \la^{n_0}$ and all $j<\la^{n}$ (use~\eqref{wzorkoc} and $j_0+j\la^{n_0}<\la^{n+n_0}$), we have
\begin{align*}
  \widetilde{u}[j_0] = \widetilde{u}[j_0 + j \la^{n_0}] = \widetilde{\theta}^{n+n_0}(\widetilde{u})[j_0 + j \la^{n_0}] = \\
  \widetilde{\theta}^{n+n_0}(\widetilde{u}[0])[j_0 + j \la^{n_0}]=
  \widetilde{\theta}^{n_0}(\widetilde{\theta}^n(\widetilde{u}[0])_j)_{j_0}.
\end{align*}
Letting $j=j_1$, this shows that $\widetilde{\theta}^{n_0}(M)_{j_0} = \widetilde{u}[j_0]$ for all $M \in \mathcal{X}, j_0 < \la^{n_0}$.
    As $\mathcal{X}$ covers $\A$ in view of~\eqref{rzut}, this shows that for all $j_0 < \la^{n_0}$, we have $\widetilde{\theta}^{n_0}(\A)_{j_0} = \widetilde{u}[j_0]$ and, therefore,
    \begin{align*}
      |\widetilde{\theta}^{n_0}(\A)_{j_0}| = c(\widetilde{\theta}),
    \end{align*}
    which concludes the proof.
 \end{proof}

\begin{Remark} It follows from Propositions~\ref{p:listas} and~\ref{p:charbij+qb}
that once $\theta$ is {\bf not} quasi-bijective,
$(X_{\widetilde\theta},S)$ is  a ``realization'' of $(H_\lambda, R)$.
\end{Remark}

\subsubsection[Joining a substitution with its synchronizing part]{Joining $\theta$ with its synchronizing part}
Define $$\ov{\mathcal{X}}:=\{(a,M):\:a\in M, M\in \mathcal{X}\}
\subset\A\times\mathcal{X}.$$
Now, in view of~\eqref{rzut}, the first compatibility condition~\eqref{s1} immediately follows. If $(a,M)\in \ov{\mathcal{X}}$ then $a\in M$, so for each $j=0,\ldots,\la-1$, we have $\theta(a)_j\in\theta(M)_j$, and in view of~\eqref{s4}, $\theta(a)_j\in \widetilde\theta(M)_j$.
Therefore, the second compatibility condition~\eqref{s2} also
follows.

Set
$$
\Theta:=
\theta\vee\widetilde{\theta}:
\ov{\mathcal{X}}\to\ov{\mathcal{X}}^{\la},$$
$$
\Theta (a,M)= (\theta(a)_0,\widetilde{\theta}(M)_0)(\theta(a)_1,
\widetilde{\theta}(M)_1)
\ldots (\theta(a)_{\la-1},\widetilde{\theta}(M)_{\la-1}).$$

\begin{Remark}\label{r:prolon1} Although, by primitivity, we can
assume that there are
$a\in\A$ and $M\in\mathcal{X}$ such that $\theta(a)_0=a$ and
$\widetilde\theta(M)_0=M$,
it may happen that~\eqref{su4} is not satisfied for $\Theta$.
However, as it does for any substitution of constant length, there exists some $k \in \N$ such that $\Theta^k =\theta^k \vee \widetilde{\theta^k}$~\footnote{\label{f:wzorek} We always have
$\Theta^k =(\theta\vee \widetilde{\theta})^k =\theta^k \vee \widetilde{\theta^k}=\theta^k\vee \widetilde\theta^k$.} satisfies~\eqref{su4}. Since
the following proposition shows that $\Theta$ satisfies \eqref{su3}, which assures us that $X_{\Theta} = X_{\Theta^k}$,
we can assume without loss of generality that $\Theta$ satisfies~\eqref{su4}.
\end{Remark}

\begin{Prop}\label{p:prim1}
Substitution $\Theta$ is primitive and $h(\Theta)=h(\theta)$.
\end{Prop}
\begin{proof}
We first show that $\Theta$ is primitive.
Let $(a_1,M_1), (a_2, M_2) \in \ov{\mathcal{X}}$. By the
definition of $\widetilde{\theta}$, it follows that there are some $k\in \N, j<\la^k$ such
that $\theta^{k}(\A)_{j} = M_2$. Hence,
$\widetilde{\theta}^{k}(M)_j = M_2$ for every $M \in \mathcal{X}$.
This ensures that there exists $a \in \A$ such that $\theta^k(a)_j = a_2$.
Using the primitivity of $\theta$, choose $k'\in \N$ and $j'<\la^{k'}$ so that $\theta^{k'}(a_1)_{j'} =
a$. We define $M := \widetilde{\theta}^{k'}(M_1)_{j'}$ and
using~\eqref{wzorkoc}, we obtain
\begin{align*}
\Theta^{k+k'}((a_1,M_1))_{j' \la^k + j} &=
\Theta^{k}(\Theta^{k'}((a_1,M_1))_{j'})_{j}\\
        &= \Theta^k((a,M))_j = (a_2,M_2)
  \end{align*}
which shows that $\Theta$ is primitive.

Let us denote by $u$ a fixed point of $\theta$, by $v$ a fixed
point of $\Theta$, where $u[0] = a_0, v[0] = (a_0,M_0)$.
According to~\eqref{defH0}, if we set
  \begin{align*}
    S_k(\theta) = \{r \geq 1: u[k+r] = u[k]\},\end{align*}
then
$$
    h(\theta) = \max\{m\geq 1: (m,\la) = 1, m \mid \gcd S_k\},$$
  for all $k\geq 0$.

We recall also that in the first part of the proof we showed that for
$(a_1,M_1)=(a_2,M_2)=(a_0,M_0)$ (and using~\eqref{wzorkoc}),
we obtain that
  \begin{align*}
    \Theta^{k+k'}((a_0,M_0))_{j' \la^k + j} = (a_0,M_0)
  \end{align*}
  for any $k' \in \N, j'<\la^{k'}$ such that
  $\theta^{k'}(a_0)_{j'} = a$. Since, $u[0,\la^{k'}-1]=\theta^{k'}(a_0)$, the above means that there exists $k \in \N, j<\la^k$ such that for all $j'$ with $u[j'] = a$, we have $v[j'\la^k + j] = (a_0,M_0)$.
  We denote by $j'_0$ the smallest such $j'$.
  Thus, we find
  \begin{align*}
    \la^k \cdot S_{j'_0}(\theta) \subset S_{j'_0 \la^k + j}(\Theta).
  \end{align*}
  This implies that $\gcd S_{j'_0 \la^k + j}(\Theta) \mid \la^k \gcd(S_{j'_0}(\theta))$, and, therefore, $h(\Theta) \mid h(\theta)$.
  Moreover, by Lemma~\ref{le:joining_height} (below),  we obtain that $h(\theta) \mid h(\Theta)$, which completes the proof.\footnote{Alternatively, $h(\theta)|h(\Theta)$ follows from the fact that $(X_{\theta},S)$ is a topological factor of $(X_{\Theta},S)$.}
\end{proof}

\begin{Remark} Note that from the measure-theoretic point of view,
the substitutional
system $(X_{\theta\vee\widetilde\theta},
\mu_{\theta\vee\widetilde\theta},S)$ (which is an extension of
$(X_\theta,\mu_\theta,S)$) is isomorphic to $(X_\theta,\mu_\theta,S)$.
Indeed, by the unique ergodicity of
$(X_{\theta\vee\widetilde\theta},S)$, it represents an ergodic
joining of $(X_\theta,\mu_\theta,S)$
and $(X_{\widetilde\theta},\mu_{\widetilde\theta},S)$ and
the latter system has discrete spectrum contained
in the discrete spectrum of $(X_\theta,\mu_\theta,S)$.
By ergodicity, the only  such a joining
is graphic. We will detail on this soon.\end{Remark}

Hence, using Proposition~\ref{p:prim1}, the above remark and
the ergodic interpretation of the column number, cf.\
Proposition~\ref{p:cextension},
we obtain the following (which is also an immediate consequence of Lemma~\ref{le:joining_height}\footnote{Indeed, as $c(\widetilde{\theta}) = 1$, Lemma~\ref{le:joining_height} shows that $c(\theta) \leq c(\Theta) \leq c(\theta)$.}):
\begin{Cor}\label{c:colnumTh}
For a substitution $\theta:\A\to\A^\lambda$, we have
$c(\Theta)=c(\theta)$.\end{Cor}

\begin{Remark}\label{r:factor1}
Note also  that $(X_{\widetilde\theta},S)$
is a topological factor of $(X_{\theta\vee\widetilde\theta},S)$
via the map:
$$
(x_n,M_n)_{n\in\Z}\mapsto (M_n)_{n\in\Z}.$$
\end{Remark}

Assume that $M\in\mathcal{X}$. As, for each $j=0,\ldots,\la-1$, the set $\{\theta(a)_j:\:a\in M\}$
has $c(\theta)$ elements, we obtain the following:

\begin{Lemma} \label{l:relbij} Substitution $\theta\vee\widetilde{\theta}$
has the following property: for each $M\in\mathcal{X}$, we have
$(\theta\vee\widetilde{\theta})(\cdot,M)_j$
is bijective for each $j=0,\ldots,\la-1$.\footnote{More precisely,
it is bijective on a set whose cardinality is $c(\theta)$, i.e., a bijection from $M$ to $\widetilde{\theta}(M)_j$.}\end{Lemma}

\begin{Remark} In this way, using joinings, we explained one
of the strategies in \cite{Mu} which consists in
representation of each automatic sequence as a
``combination'' of the synchronized and relatively invertible parts.
\end{Remark}

\subsubsection[Description of the subshift given by the new substitution]{Description of the subshift
$(X_{\theta\vee\widetilde\theta},S)$}\label{ss:c}
Each point of the space $X_{\theta\vee\widetilde\theta}$ is
of the form $(x_n,M_n)_{n\in\Z}$, where $\ov{M}:=(M_n)\in
X_{\widetilde\theta}$ with $M_n\in\mathcal{X}$,
and $x=(x_n)\in X_\theta$, so for each $n\in\Z$, $x_n\in M_n\subset\A$.
More than that, like every substitutional system, such a point must
have its (unique) $\la^t$-skeleton structure $(j_t)_{t\geq 1}$
on which we will detail more. First of all, by the definition
of $\theta\vee\widetilde\theta$, $(j_t)$ must be the
$\la^t$-skeleton structure of both $\ov M$ and $x$
(we strongly use here that the skeleton structure is unique;
for that, we need that the substitution $\tilde\theta$
is non-trivial, in other words the argument makes no sense
if $\theta$ is quasi-bijective).

So $(j_t)_{t\geq1}$ is the $\la^t$-skeleton structure
of $\ov M\in X_{\widetilde\theta}$: for each $t\geq1$, we have
$$
\ov{M}[-j_t+s\la^t,-j_t+(s+1)\la^t-1]=\widetilde\theta^t(R_s)$$
for each $s\in\Z$ and some $R_s\in\mathcal{X}$.
Fix $t\geq 1$ and consider $s=0$. We have
$$
\ov{M}[-j_t,-j_t+\la^t-1]=\widetilde\theta^t(R_0)$$
and $\mathcal{X}\ni R_0=\{r_0,\ldots,r_{c-1}\}$, where $c=c(\theta)$
and $r_j\in\A$. Hence
\beq\label{matrix}
\widetilde\theta^t(R_0)=
\begin{array}{lllll}
\theta^t(r_0)_0&\ldots&\theta^t(r_0)_{j_t}&\ldots&
\theta^t(r_0)_{\la^t-1}\\
\theta^t(r_1)_0&\ldots&\theta^t(r_1)_{j_t}&\ldots&
\theta^t(r_1)_{\la^t-1}\\
\ldots&\ldots&\ldots&\ldots&\ldots\\
\theta^t(r_{c-1})_0&\ldots&\theta^t(r_{c-1})_{j_t}&\ldots&
\theta^t(r_{c-1})_{\la^t-1},\end{array}
\eeq
where the columns ``represent'' sets in $\mathcal{X}$. We look at $\widetilde\theta^t(R_0)_{j_t}$ the
$j_t$-th element of $\widetilde\theta^t(R_0)$ which is
``represented'' by the $j_t$-th column of the matrix above and
is equal to $M_0$ (the zero coordinate of $\ov M$). Now, $x_0\in M_0$,
hence $x_0=\theta^t(r_\ell)_{j_t}$ for a unique $0\leq \ell\leq c-1$.
Moreover, $x[-j_t,-j_t+\la^t-1]=\theta^{t}(r_\ell)$. Hence,
unless $j_t=\la^t-1$, $x_0$ determines $x_1$.

Note that, we can reverse this reasoning in the case:
\beq\label{expand}
\min(j_t,\la^t-j_t-1)\to\infty.\eeq
Indeed, in this case, given $\ov{M}$, we can choose
$x_0\in M_0$ arbitrarily, and then successively fill in the
second coordinate by placing there the $\ell$-th row in the
matrix~\eqref{matrix}, where $x_0=\theta^t(r_\ell)_{j_t}$.
This shows that over all $\ov{M}\in X_{\widetilde\theta}$
we have $c$ points $(x,\ov{M})\in X_{\theta\vee\widetilde\theta}$
whenever the $\la^t$-skeleton of $\ov{M}$ satisfies~\eqref{expand}.
What remains are points for which $$
\mbox{either $j_{t+1}=j_t$ or $\la^{t+1}-j_{t+1}-1=\la_t-j_t-1$
for $t\geq T$}.$$
When projected down to the maximal equicontinuous factor,
this condition defines a countable set, in particular, of
(Haar) measure zero. For the unique measure $\mu_{\widetilde\theta}$
we have hence a.e.\ a $c$-point extension. For $\ov{M}$ which do
not satisfy~\eqref{expand}, we need (at most)
two coordinates for $x$ to determine all other,
so the fibers have at most $c^2$ elements.

\begin{Example} Let $\A=\{a,b,c,d\}$.
  We consider the following substitution $\theta: \A \to \A^4$:
  \begin{align*}
    \theta(a) &= adda,\;
    \theta(b) = bccb\\
    \theta(c) &= abbc,\;
    \theta(d) = baad.
  \end{align*}
  It is not hard to see that $\theta$ is pure and $c=2, \mathcal{X} = \{M_0, M_1\}$, where $M_0 = \{a,b\}$ and $M_1 = \{c,d\}$. Moreover,
  \begin{align*}
    \til{\theta}(M_0) &= M_0 M_1 M_1 M_0\\
    \til{\theta}(M_1) &= M_0 M_0 M_0 M_1.
  \end{align*}
Let us now consider the block $\ov{M}^{(0)}=M_0M_0\in \mathcal{X}^2$, which we write as $\ov{M}^{(1)}_{-1}\ov{M}^{(1)}_0$. Acting on it by $\widetilde{\theta}$, we obtain
$$\ov{M}^{(1)}=\widetilde\theta(M_0)\widetilde\theta(M_0)=\ov{M}^{(1)}_{-4}\ldots
  \ov{M}^{(1)}_{-1}\ov{M}^{(1)}_0\ldots\ov{M}^{(1)}_3.$$
By iterating this procedure and passing to the limit, we obtain a two-sided sequence which is a fixed point for $\widetilde\theta$ and which we denote by
$\overline{M}=\til{\theta}^{\infty}(M_0).\til{\theta}^{\infty}(M_0)$, the ``dot'' indicating the zero position of the sequence. A similar procedure can be made on the $\theta$-side, starting with $a.a$, $a.b$, $b.a$ and $b.b$. It is not hard to see that, up to a natural rearrangement of coordinates,
the following four points $(\theta^{\infty}(a).\theta^{\infty}(a),\ov{M})$, $(\theta^{\infty}(a).\theta^{\infty}(b),\ov{M})$, $(\theta^{\infty}(b).\theta^{\infty}(a),\ov{M})$, $(\theta^{\infty}(b).\theta^{\infty}(b),\ov{M})$
are members of $X_{\theta\vee\widetilde\theta}$. It follows that the fiber over $\ov{M}$ has four points.
\end{Example}

\subsection{Toward a skew-product measure-theoretic representation -
making the relative invertibility clearer}\label{s:rename}
In this part we want to rename the alphabet of $\Theta$
to get a new substitution, which makes the
invertible part easier to handle. Namely, our new alphabet
will be the set $\{0,\ldots,c-1\}\times\mathcal{X}$, where $c=c(\theta)$. The only
thing we need is to give a ``good'' identification
of $\{0,\ldots,c-1\}\times M$
with $\{(a,M):\:a\in M\}$ (for $M\in\mathcal{X}$).
We start by giving a classical example.

\begin{Example} [Rudin-Shapiro sequence] \label{ex:ex2}
  We consider the Rudin-Shapiro substitution $\theta$ defined by
$\A=\{a,b,c,d\}$ and  \begin{align*}
    a \mapsto ab, \quad &d \mapsto dc\\
    b \mapsto ac, \quad &c \mapsto db.
  \end{align*}
  We find that $\mathcal{X} = \{ \{a,d\},\{b,c\}\}$ and
  consequently for $\Theta$, we have:
  \begin{align*}
    (a,\{a,d\}) \mapsto (a,\{a,d\})(b,\{b,c\}), \quad &(d,\{a,d\})
    \mapsto (d,\{a,d\})(c,\{b,c\})\\
    (b,\{b,c\}) \mapsto (a,\{a,d\})(c,\{b,c\}), \quad &(c,\{b,c\})
    \mapsto (d,\{a,d\})(b,\{b,c\}).
  \end{align*}
  Let us now fix a partial ordering on $\A$ that is complete
  on every $M \in \mathcal{X}$, e.g., $a < b$, $c < d$. Thus, we can
  identify $(a,\{a,d\}) \cong (0,\{a,d\})$ and $(d,\{a,d\})
  \cong (1,\{a,d\})$ as $a$ is the smallest element of
  $\{a,d\}$ and $d$ is the second smallest element. This gives the
  new substitution $\widetilde{\Theta}: \{0,\ldots,c-1\}
  \times\mathcal{X} \to (\{0,\ldots,c-1\}\times\mathcal{X})^{\la}$:
  which in our example is given by:
  \begin{align*}
    (0,\{a,d\}) \mapsto (0,\{a,d\})(0,\{b,c\}), \quad &(1,\{a,d\})
    \mapsto (1,\{a,d\})(1,\{b,c\})\\
    (0,\{b,c\}) \mapsto (0,\{a,d\})(1,\{b,c\}),
    \quad &(1,\{b,c\}) \mapsto (1,\{a,d\})(0,\{b,c\}).
  \end{align*}

  Note that $\til{\theta}$ is very simple in this example,
  as identifying $\{a,d\}$ with $M_0$ and $\{b,c\}$ with $M_1$,
  gives $\til{\theta}(M_0)=M_0 M_1$ and $\til{\theta}(M_1)=M_0 M_1$,
  so we obtain a periodic sequence.
  We obtain the classical Rudin-Shapiro sequence by considering the
  fixpoint of $\widetilde{\Theta}$ and then applying
  the projection $(i, M) \mapsto i$. (We emphasize that the classical Rudin-Shapiro sequence itself is not a fixpoint of a substitution.)
\end{Example}

We will be dealing with $\theta$ which is not quasi-bijective
(which, via Propositions~\ref{p:listas} and~\ref{p:charbij+qb}, guarantees that $(X_{\widetilde\theta},S)$ ``captures''
the whole discrete spectrum of $(H_\lambda,m_{H_\lambda},R)$).
We repeat the above construction:   fix a partial order on $\A$
that is complete on every $M \in \mathcal{X}$ and identify
$(a,M)$ with $(j,M)$ if the ordering of the elements in $M$ according
to the ordering on $\A$ restricted to $M$ yields $a$ as the
$(j+1)$-th smallest element and rewrite instructions using the alphabet.
This yields the substitution
$$
\widetilde\Theta:\{0,\ldots,c-1\}\times\mathcal{X}\to
(\{0,\ldots,c-1\}\times\mathcal{X})^\lambda.$$
The second coordinate is still $\widetilde{\theta}$ which
corresponds to the synchronizing part and can be defined independently
of the first coordinate.
The first coordinate (which now depends on the second coordinate),
gives the ``invertible part''.

\begin{Lemma} $\widetilde\Theta$ is primitive,
 $h(\widetilde{\Theta}) = h(\theta)$ and
 $c(\widetilde{\Theta})=c(\theta)$.
\end{Lemma}
\begin{proof}
As $\widetilde\Theta$ is obtained by renaming the alphabet
of $\Theta$, this follows immediately from Proposition~\ref{p:prim1}
and Corollary~\ref{c:colnumTh}.
\end{proof}

\begin{Remark}\label{r:factor2}
A point in the space $X_{\widetilde\Theta}$ is of the form
$(y_n,M_n)_{n\in\Z}$ with $y_n\in\{0,\ldots,c-1\}$ and there
is an equivariant map between $X_{\widetilde\Theta}$
and $X_{\theta\vee\widetilde\theta}$ given by
\beq\label{c-ext}(y_n,M_n)\mapsto (z_n,M_n),\eeq
where $z_n\in M_n$ is the $(y_n+1)$-th letter in the ordering on $M_n$.
Composing this map with the projection on the first coordinate
yields the topological factor $(X_{\theta},S)$.\end{Remark}

The projection of \eqref{c-ext} on the second coordinate yields
the factor $(X_{\widetilde\theta},S)$. According to
Subsection~\ref{ss:c}, we can now represent
$(X_{\widetilde\Theta},\mu_{\widetilde\Theta},S)$
as a skew product over $(X_{\widetilde\theta},
\mu_{\widetilde\theta},S)$ using the following: first
$$
\mbox{identifying $(i_n,M_n)_{n\in\Z}$ with
$((M_n)_{n\in\Z},i_0)\in X_{\widetilde\theta}\times\{0,\ldots,c-1\}$}
$$
(which is possible for all $\ov{M}=(M_n)_{n\in\Z}$ satisfying~\eqref{expand},
cf.\ Subsection~\ref{ss:c}) and then setting
\beq\label{skew}
(\ov{M},i_0)\mapsto (S\ov{M},\sigma_{\ov{M}}(i_0)),
\eeq
where $\sigma_{\ov{M}}$ (in fact, it is $\sigma_{M_0}$)
is a permutation of $\{0,\ldots,c-1\}$
determined by $M_0$: we come back to the matrix~\eqref{matrix}
in which $M_0$ is represented as the $j_t$-th column.
The next column is representing $M_1$, and if we have $z_0$
(which belongs to $M_0$ and corresponds to $i_0$)
in the $i_0$-th row then $z_1$
is just the next element in the same row. Now, the elements in
the $j_t$-th column have their names in $\{0,\ldots,c-1\}$
and this is the same for $(j_t+1)$-st column. Hence $\sigma_{\ov{M}}$
is a permutation of $\{0,\ldots,c-1\}$ sending  an element in the
$j_t$-th column to the neighboring one in the next column.

\begin{Remark}\label{r:firstconcl} As from the measure-theoretic point
of view, $(X_{\widetilde{\Theta}},S)$ is still $(X_\theta,S)$ (and
$(X_{\widetilde\theta},S)$ is the same as $(H_\lambda,R)$),
the above formula~\eqref{skew}
clears up Remark~9.1 in \cite{Qu}, p.~229, about
the form of a cocycle representing $(X_\theta,\mu_\theta, S)$ as a skew product
over $(H_\lambda,m_{H_\lambda},R)$.

Coming back to our main problem, note that $(X_{\widetilde\Theta},S)$ has
$(X_\theta,S)$ as its {\bf topological} factor and it has also
$(X_{\widetilde\theta},S)$ ``representing'' measure-theoretically
$(H_\lambda,R)$ as its {\bf topological} factor, but {\bf still}
we do not
know whether we have decomposition~\eqref{dobre} for each $F\in
C(X_{\widetilde\Theta})$. To assure it, we will need another
extension.\end{Remark}

At the end of this section we want to discuss one particular ordering that has useful properties.

We start by taking an arbitrary complete order $<_{M_0}$ on $M_0$ such that $a_0$ is the minimum. By Proposition~\ref{p:listas}, $\widetilde{\theta}$ is synchronizing and primitive, thus we find $k_0',j_0'$ such that $\widetilde{\theta}^{k_0'}(M)_{j_0'} = M_0$ for all $M \in \mathcal{X}$.
As $\theta^{k_0'}_{j_0'}$ permutes the elements of $M_0$, we find by iterating $\theta^{k_0'}_{j_0'}$, some $k_0, j_0<\la^{k_0}$ such that $\theta^{k_0}(a)_{j_0} = a$ for all $a\in M_0$ and $\widetilde{\theta}^{k_0}(M)_{j_0} = M_0$ for all $M \in \mathcal{X}$.

Using~\eqref{rzut}, we extend now the ordering on $M_0$ to a partial (but not complete in general, as below, $a\neq b$ remain incomparable if $\theta^{k_0}(a)_{j_0}=\theta^{k_0}(b)_{j_0)}$) ordering on $\A$ by $$a<b \;\text{ if and only if }\; \theta^{k_0}(a)_{j_0} <_{M_0} \theta^{k_0}(b)_{j_0}.$$
This gives a complete ordering on any $M \in \mathcal{X}$ and, thus, allows us to define $\widetilde{\Theta}$.
With this ordering we obtain directly that for all $M \in \mathcal{X}$ and  $i=0,1,\ldots,c-1$, we have
\begin{align}\label{eq:synch}
  \til{\Theta}^{k_0}(i,M)_{j_0} = (i,M_0)
\end{align}
which will be useful later.

\begin{Remark}\label{r:wzorekprime} As the order on each $M\in\mathcal{X}$ is fixed, we have $\widetilde{\Theta}^k=\widetilde{\Theta^k}$ (cf.\ footnote~\ref{f:wzorek}).\end{Remark}

\subsection{Column number, height and pure base -- revisited}\label{s:chpbrevisited}
In this part we come back to connections between the column number and the height of primitive substitutions.
We start with simple observations about substitution joinings.
\begin{Lemma}\label{le:joining_height}
  Let $\theta:\A\to\A^\lambda$ and $\zeta:\B\to\B^\lambda$ be substitutions of length $\la$ fulfilling \eqref{su3}, \eqref{su4},~\eqref{su6} and~\eqref{su7}, but are not necessarily pure. For some $\mathcal{A}\subset\A\times\B$, let $\Sigma = \theta \vee \zeta:\mathcal{A}\to\mathcal{A}^\lambda$ be a substitutional joining (in particular, we assume that $\Sigma$ is primitive). Then,
  \begin{align}\label{eq:joining_height1}
    {\rm lcm} (h(\theta),h(\zeta)) \mid h(\Sigma),
  \end{align}
  which also yields a lower bound for $h(\Sigma)$. Furthermore,
  \begin{align*}
    \max(c(\theta),c(\zeta)) \leq c(\Sigma) \leq c(\theta) \cdot c(\zeta),
  \end{align*}
  which, in particular, yields
  lower and upper bounds for the column number of $\Sigma$.
\end{Lemma}
\begin{proof}
  One sees directly that $S_k(\Sigma) = S_k(\theta) \cap S_k(\zeta)$ and, therefore, $g_k(\theta) \mid g_k(\Sigma)$ and $g_k(\zeta) \mid g_k(\Sigma)$ which implies \eqref{eq:joining_height1}.\footnote{One can also see this result dynamically, as both $(X_\theta,S)$ and $(X_\zeta,S)$ are topological factors of $(X_\Sigma,S)$.}

  As the projections of $\mathcal{A}$ on $\A$ and $\B$ are full,
  the lower bound for $c(\Sigma)$ is obvious.
  Furthermore, if $|\theta^{k}(\A)_j|=c(\theta)$, $|\zeta^{k'}(\B)_{j'}|=c(\zeta)$ then (using~\eqref{wzorkoc}), we have
  \begin{align*}
    \Sigma^{k'+k}(\mathcal{A})_{j\la^{k'}+j'} &\subseteq \Sigma^{k'+k}(\A \times \B)_{j \lambda^{k'} + j'} = \Sigma^{k'}(\Sigma^{k}(\A \times \B)_j)_{j'}\\
    &\subseteq \Sigma^{k'}(\theta^k(\A)_j \times \B)_{j'} \subseteq \theta^{k'}(\theta^k(\A)_j)_{j'} \times \zeta^{k'}(\B)_{j'}.
  \end{align*}
  This shows directly that $|\Sigma^{k'+k}(\mathcal{A})_{j\la^{k'}+j'}| \leq c(\theta) \cdot c(\zeta)$ and completes the proof.
\end{proof}

Let us now pass to the proof of Lemma~\ref{le:joining_height0}:

{\bf Proof of Lemma~\ref{le:joining_height0}} (We assume that the ordering on $\A$ is as the one described at the end of Subsection~\ref{s:rename}.)
We have already seen that $c := c(\theta) = c(\widetilde{\Theta})$ and $h := h(\theta) = h(\widetilde{\Theta})$, so we consider without loss of generality the substitution $\widetilde{\Theta} : (\{0,\ldots,c-1\} \times \mathcal{X}) \to (\{0,\ldots,c-1\} \times \mathcal{X})^{\la}$.
It follows directly from the basic properties of the height that we discussed at the beginning of the paper that there exists a 1-coding $f: \mathcal{A} := (\{0,\ldots,c-1\} \times \mathcal{X}) \to \{0,\ldots,h-1\}$ such that the fixed point $u = \widetilde{\Theta}(u)$ is mapped to the periodic sequence $$01\ldots (h-1)01\ldots(h-1)\ldots.$$
It follows by a simple computation (cf.~\eqref{wzorkoc}) that
\begin{align*}
  u[j' \lambda^k+j] = \widetilde{\Theta}^k(u[j'])_{j}
\end{align*}
whenever $j'\geq0$ and $j<\la^k$.
As $f(u[\ell]) \equiv \ell \bmod h$ for each $\ell\geq0$, it follows that (letting $j'$ be so that $(i,M)=u[j']$ and $\ell=j'\la^k+j$)
\begin{align*}
  f(\widetilde{\Theta}^k(i,M)_j) = f(\widetilde\Theta^k(u[j'])_j)=
  f(u[j'\la^k+j])\equiv j'\la^k+j \bmod h.\end{align*}
By the same token, $f(i,M)\equiv j' \bmod h$, so finally
\beq\label{rowno1}
 f(\widetilde{\Theta}^k(i,M)_j)\equiv f(i,M) \la^k + j \bmod h
\eeq
for each $k\geq0$ and $j<\la^k$.

Furthermore, we know by the construction of $\widetilde{\Theta}$ that there exist $k_0, j_0$ such that $\widetilde{\Theta}^{k_0}(i,M)_{j_0} = (i,M_0)$ for every $i \in \{0,\ldots,c-1\}$ and $M \in \mathcal{X}$, as noticed in~\eqref{eq:synch}.
Thus, in view of~\eqref{rowno1} (for $j=j_0$ and $k=k_0$), for all $i \in \{0,\ldots,c-1\}$ and $M \in \mathcal{X}$, we obtain that
\begin{align*}
  (f(i,M) \la^{k_0} + j_0 \bmod h) = f(\widetilde{\Theta}^{k_0}(i,M)_{j_0}) = f(i,M_0).
\end{align*}
This implies that $f(i,M)$ does not depend on $M$ (as $\lambda$ is a multiplicatively invertible element in the ring $\Z/h\Z$), so  $f(i,M)=f(i,M_0)$ and $f$ only depends on the first coordinate. Therefore, we denote $f'(i) := f(i, M_0)$.

As the second coordinate in $\widetilde\Theta$ equals $\widetilde\theta$ is independent of the first coordinate and $\til{\Theta}(.,M)_j$ is a bijection from $\{0,\ldots,c-1\}$ to itself, we obtain the following formula for the incidence matrix $M(\til{\Theta})$,
\beq\label{niezalezy}
  \sum_{i'<c}  M(\til{\Theta})_{(i,M),(i',M')} = M(\til{\theta})_{M,M'}.
\eeq
We now claim that
\beq\label{claim10}
\delta_{(i,M)} = \delta_M \cdot \frac{1}{c},\eeq
i.e.\ the density of $(i,M)$ for $\til{\Theta}$ is equal to $e_{i,M}:=\frac1c\delta_M$. Indeed, using the definition of $e_{i,M}$, \eqref{niezalezy} and~\eqref{wzor100} (for $\widetilde\theta$), we obtain
\begin{align*}
  \sum_{(i',M') \in \{0,\ldots,c-1\} \times \mathcal{X}} e_{(i',M')} \cdot M(\til{\Theta})_{(i,M),(i',M')} &= \sum_{M' \in \mathcal{X}} \frac{1}{c} \delta_{M'} \cdot \sum_{i'<c} M(\til{\Theta})_{(i,M),(i',M')} \\
  &= \sum_{M' \in \mathcal{X}} \frac{1}{c} \delta_{M'} \cdot M(\til{\theta})_{M,M'}\\
  &= \la \cdot \frac{1}{c} \delta_M=\la\cdot e_{i,M}.
\end{align*}
But (by the Perron-Frobenius theorem) there is a unique solution to the above, given by $\delta_{i,M}$, so $e_{i,M}=\delta_{i,M}$ and the claim~\eqref{claim10} follows.

Clearly, each letter $j\in\{0,\ldots,h-1\}$ appears in $01\ldots(h-1)01\ldots$ with density $\frac{1}{h}$ and, recall that, this periodic sequence is the image of $u$ under $f$. Using this and~\eqref{claim10}, we obtain
\begin{align}\label{eq:c_h}
  1/h = \sum_{(i,M): f(i,M) = j} \delta_{(i,M)} = \sum_{i: f'(i) = j} \sum_{M} \frac{1}{c} \delta_M =  \frac{1}{c}|\{i<c: f'(i) =j\}|.
\end{align}

Equation~\eqref{eq:c_h} shows directly that $h | c$ and $\#\{i<c: f'(i) = j\}$ is constant and equals $c/h$.
It remains to show that the column number of the pure base of $\theta$ equals $c/h$. For this aim, let us first
notice that we can also view $f$ as well defined on $\mathbb{A}$ and that for all $M \in \mathcal{X}$ and $j<h$, we have
\beq\label{liczba}\#\{a \in M: f(a) = j\} = c/h.\eeq

Recall also the construction of the pure basis $\theta^{(h)}$ of~$\theta$. It is a
substitution defined  over the alphabet $\mathbb{A}^{(h)}$ which is the set $\{u[mh(\theta),mh(\theta)+h-1]:\:m\geq0\}$ (which is of course finite). Then, we  define $\theta^{(h)}:\mathbb{A}^{(h)}\to(\mathbb{A}^{(h)})^{\la}$ by setting
\begin{align*}
\theta^{(h)}(w)_j=\theta(w)[j h(\theta),(j+1) h(\theta)-1]
\end{align*}
for $j=0,\ldots,\la-1$.

If $w=(w_0,\ldots,w_{h-1}) \in \mathbb{A}^{(h)}$, then there exists some $m$ such that $a_i = u[mh + i]$ for $i=0,\ldots,h-1$.
It follows that
\beq\label{wzora0}
f(w_i)=f(u[mh+i]) \equiv mh+i \equiv i \mod h\eeq
which gives $f(w_i) = i$.
For each $k\geq1$ and $j<\la^k$, we have
\beq\label{wzora1}
  (\theta^{(h)})^k(w)_j = \theta^k(w)[j h(\theta), (j+1) h(\theta)-1].
\eeq
Choose $k>1$ such that $\la^k \equiv 1 \bmod h$ and $\la^k > (j_0+1) h$. Now, $j_0\frac{\la^k-1}h+j_0<\la^{k+k_0}$.
In view of~\eqref{wzora1} (for $k+k_0$ and $j_0(\la^k-1)/h+j_0$) and then~\eqref{wzorkoc} (which we may use as $(j_0+1)(h-1)<\la^k$), we obtain that
\begin{align*}
  (\theta^{(h)})^{k+k_0}&(w_0,\ldots,w_{h-1})_{j_0 (\la^k-1)/h +j_0}\\
   &= \theta^{k+k_0}(w_0,\ldots,w_{h-1})[j_0 \la^k + j_0 (h-1), j_0 \la^k + (j_0+1)(h-1)] \\
   &=  \theta^k(\theta^{k_0}(w_0)_{j_0})[j_0(h-1),(j_0+1)(h-1)].
\end{align*}
It follows from this calculation that the column number of $\theta^{(h)}$ is not larger than the number of $\theta^{k_0}(a_0)_{j_0}$, when we run over all possible $(w_0,\ldots,w_{h-1})\in\A^{(h)}$. But,
as before (cf.~\eqref{rowno1}), $f(\theta^k(a)_j)=f(a)\la^k+j\; {\rm mod}~h$ (for each $a\in\A$). It follows that $f(\theta^{k_0}(w_0)_{j_0}) \equiv j_0 \bmod h$ (as $f(w_0)=0$, see~\eqref{wzora0}).
Moreover, by~\eqref{eq:synch}, $\theta^{k_0}(a)_{j_0} \in M_0$ (for each $a\in\A$). This shows that $c(\theta^{(h)})$ is not larger than the number of $b\in M_0$ such that $f(b)\equiv j_0$~mod~$h$ and hence, in view of~\eqref{liczba}, $c(\theta^{(h)})\leq c(\theta)/h(\theta)$.

Next, we want to give a lower bound for $c(\theta^{(h)})$.
We fix any $k > 0$ and $j<\la^k$. Denote $i = \lfloor \frac{jh}{\la^k} \rfloor$. In view of~\eqref{hei2}, it follows that we have for any $(w_0,\ldots,w_{h-1}) \in \mathbb{A}^{(h)}$,
\begin{align*}
  (\theta^{(h)})^k(w_0,\ldots,w_{h-1})_j = (v_0,\ldots,v_{h-1}),
\end{align*}
where
\begin{align*}
  v_0 = \theta^{k}(w_i)_{jh - i \la^k}.
\end{align*}
Moreover, $f(w_i)=i$ (see~\eqref{wzora0}).
Thus, we find
\begin{align*}
  \#(\theta^{(h)})^k(\mathbb{A}^{(h)})_j \geq \# \theta^{k}(\mathbb{A}_i)_{jh - i\la^k},
\end{align*}
where $\mathbb{A}_i := \{a \in \mathbb{A}: f(a) = i\}$. Finally, we recall the relative bijectivity: for each $M\in\mathcal{X}$ and any different $a,a'\in M$, we have
$\theta^k(a)_{j'} \neq \theta^k(a')_{j'}$ for each $j'<\la^k$, see the last phrase before Lemma~\ref{l:relbij}.
This shows that
\begin{align*}
  c(\theta^{(h)}) \geq \#\{a \in M_0: f(a) = i\} = \frac{c(\theta)}{h(\theta)},
\end{align*}
which completes the proof.
\bez

\begin{Cor} The dynamical system $(X_\theta,\mu_\theta,S)$
 given by (primitive) $\theta$ has purely discrete spectrum if and only if $h(\theta) = c(\theta)$.
\end{Cor}

\subsection{Toward a group extension}\label{sec:group}
We will now aim at an extension  of the  substitution $\widetilde\Theta$ to a substitution
$\widehat\Theta$ (of length $\la$) over the alphabet
$G \times \mathcal{X}$, where $G$ is a subgroup of the group $\mathcal{S}_c$ of bijections on $\{0,\ldots,c-1\}$.
We would like to note that in case of bijective substitutions a  similar idea was
used in~\cite{Qu} to study the spectrum of the associated system.

Recall that for each $j=0,\ldots,\la-1$ and $M\in\mathcal{X}$,
$\widetilde{\Theta}(\cdot,M)_j$ is bijective in view of Lemma~\ref{l:relbij}
and the definition of $\widetilde\Theta$.
We define the corresponding permutation $\sigma_{M,j} \in \mathcal{S}_c$ by $\sigma_{M,j}(m) = n$ iff $\widetilde{\Theta}(n,M)_j =
(m, \widetilde{\theta}(M)_j)$. This allows us to define the aforementioned group $G$ as the group generated by the $\sigma_{M,j}$, i.e.
\begin{align}
  G := <\sigma_{M,j}: M \in \mathcal{X}, j \in \{0,\ldots,\la-1\}>.
\end{align}

Finally, we define $\widehat{\Theta}: G \times \mathcal{X} \to (G \times \mathcal{X})^{\la}$ by setting
\begin{align}\label{dejFINAL0}
  \widehat{\Theta}(\sigma,M)_j := (\sigma \circ \sigma_{M,j}, \widetilde{\theta}(M)_j),
\end{align}
for $j = 0,\ldots,\la-1$.

Repeating the same with $\theta$ replaced with $\theta^k$ (cf.~Remark~\ref{r:wzorekprime}),  we define the permutations $\sigma^{(k)}_{M,j}$ of $\{0,\ldots,c-1\}$ so that
\begin{align}\label{eq:connection}
  \sigma^{(k)}_{M,j}(m) = n \text{ if and only if } \widetilde{\Theta}^{k}(n,M)_j = (m, \til{\theta}^{k}(M)_j).
\end{align}
This is a very important formula, as it highlights the connection between $\widetilde{\Theta}$ and $\widehat{\Theta}$.
Furthermore,
analogously to~\eqref{dejFINAL0}, we find that
\begin{align}\label{wzornak}
  \widehat{\Theta}^{k}(\sigma, M)_j = (\sigma \circ \sigma^{(k)}_{M,j}, \til{\theta}^k(M)_j),
\end{align}
for $j=0,\ldots,\la^k-1$.

It follows directly that $\sigma^{(1)}_{M,j} = \sigma_{M,j}$ and from~\eqref{wzorkoc} (applied to $\widehat\Theta$) that
\begin{align}\label{coc11}
  \sigma^{(k_1+k_2)}_{M, j_1 \la^{k_2} + j_2} = \sigma^{(k_1)}_{M,j_1} \circ \sigma^{(k_2)}_{\til{\theta}^{k_1}(M)_{j_1}, j_2}
\end{align}
for $j_1<\la^{k_1}$ and $j_2<\la^{k_2}$.

If $\widetilde\theta(M_0)_0=M_0$, then clearly $\widetilde\theta^k(M_0)_0=M_0$ and applying~\eqref{coc11} for $(\sigma,M_0)$, with $j_1=j_2=0$, we obtain $\sigma_{M_0,0}^{(k)}=\left(\sigma_{M_0,0}\right)^k$.
Hence,
by possibly passing to an iterate of $\widehat{\Theta}$, we can assume that
\begin{align}\label{eq:hat_prolongable}
  \widehat{\Theta}(\sigma, M_0)_0 = (\sigma, M_0)
\end{align}
for each $\sigma\in G$.

We recall~\eqref{eq:synch}, which shows the existence of some $k_0 \in \N$ and $j_0 < \la^{k_0}$ such that $\widetilde{\Theta}^{k_0}(i,M)_{j_0} = (i,M_0)$ for all $i<c$ and $M \in \mathcal{X}$. This gives directly, by~\eqref{eq:connection}, that $\sigma_{M,j_0}^{(k_0)} = id$, or in other words
\begin{align}\label{eq:hat_synchronizing}
  \widehat{\Theta}^{k_0}(\sigma, M)_{j_0} = (\sigma, M_0)
\end{align}
holds for all $\sigma \in G, M \in \mathcal{X}$.

The next few lemmas will be used to show that $\widehat{\Theta}$ is primitive.
We start by giving some kind of a dual statement to~\eqref{eq:hat_synchronizing}.
\begin{Lemma}\label{l:kaem}
    For every $M \in \mathcal{X}$ there exist $k_M \in \N$ and $j_M<\la^{k_M}$ such that $\widehat{\Theta}^{k_M}(\sigma,M_0)_{j_M} = (\sigma,M)$
    for each $\sigma\in G$.
\end{Lemma}
\begin{proof}
We fix $M \in \mathcal{X}$. As $\tilde{\theta}$ is primitive, we can find $k \in \N,j<\la^k$ such that $\til{\theta}^{k}(M_0)_j = M$, which gives
$\widehat{\Theta}^{k}(\sigma,M_0)_j = (\sigma \circ \sigma^{(k)}_{M,j},M)$ for all $\sigma \in G$.
In view of~\eqref{wzorkoc}, we obtain that
  \begin{align}\label{eq:returnToM}
  \begin{split}
    \widehat{\Theta}^{k_0 + k}(\sigma,M_0)_{j \la^{k_0} + j_0} = \widehat{\Theta}^{k_0}(\widehat{\Theta}^k(\sigma,M_0)_j)_{j_0}\\ = \widehat{\Theta}^{k_0}(\sigma \circ \sigma^{(k)}_{M,j}, M)_{j_0} = (\sigma \circ \sigma^{(k)}_{M,j}, M_0).
  \end{split}
  \end{align}
  Set $\ell :={\rm ord}(\sigma^{(k)}_{M,j})$. Now, by iterating~\eqref{eq:returnToM}, using~\eqref{wzorkoc}, we obtain that
  \begin{align*}
    \widehat{\Theta}^{k'}(\sigma,M_0)_{j'} = (\sigma \circ (\sigma^{(k)}_{M,j})^{\ell-1}, M_0),
  \end{align*}
  where $k' = (\ell-1) (k_0+k)$ and $j' = \la^{(\ell-2)(k_0+k)}(j \la^{k_0} + j_0) + \ldots + \la^{k_0+k}(j \la^{k_0} + j_0) + j\la^{k_0} + j_0$.
  It follows that
  \begin{align*}
    \widehat{\Theta}^{k + k'}(\sigma,M_0)_{j + \la^{k} j'} &= \widehat{\Theta}^k(\widehat{\Theta}^{k'}(\sigma,M_0)_{j'})_{j} = \widehat{\Theta}^k(\sigma \circ (\sigma^{(k)}_{M,j})^{\ell-1},M_0)_{j}\\
    &= (\sigma \circ (\sigma^{(k)}_{M,j})^{\ell}, M) = (\sigma,M)
  \end{align*}
  which completes the proof.
\end{proof}

Next, we find a better description for $G$.

\begin{Lemma}
We have
    \begin{align*}
      G = <\sigma^{(k)}_{M_0,j}: k\in \N, j<\la^k, \til{\theta}^k(M_0)_j = M_0>.
    \end{align*}
  \end{Lemma}
  \begin{proof}
By~\eqref{wzorkoc}, it is clear that any $\sigma^{(k)}_{M_0,j} \in G$.
Thus, it remains to write any $\sigma_{M,j}$ as $\sigma^{(k)}_{M_0,j'}$, where $\til{\theta}^k(M_0)_{j'} = M_0$.
Using~\eqref{wzorkoc}, Lemma~\ref{l:kaem}, \eqref{dejFINAL0} and finally~\eqref{eq:hat_synchronizing}, we have
    \begin{align*}
    \widehat{\Theta}^{k_0 + 1 + k_M}(\sigma,M_0)_{j_0 + \la^{k_0} j + \la^{k_0 + 1} j_M} &=
    \widehat{\Theta}^{k_0}(\widehat{\Theta}(\widehat{\Theta}^{k_M}(\sigma, M_0)_{j_M})_{j})_{j_0}\\
      &= \widehat{\Theta}^{k_0}(\widehat{\Theta}(\sigma, M)_{j})_{j_0}\\
        &= \widehat{\Theta}^{k_0}(\sigma \circ \sigma_{M,j}, \til{\theta}(M)_j)_{j_0}\\
        &= (\sigma \circ \sigma_{M,j}, M_0),
  \end{align*}
  which, by~\eqref{wzornak},  ends the proof.
  \end{proof}

  This allows us to give the final description for $G$.
  \begin{Lemma}\label{le:Grep}
    We have
    \begin{align*}
      G = \{\sigma^{(k)}_{M_0,j}:\: k\in \N,\; j<\la^k, \; \til{\theta}^k(M_0)_j = M_0\}.
    \end{align*}
  \end{Lemma}

  \begin{proof}
    It just remains to show that
    \begin{align*}
      \{\sigma^{(k)}_{M_0,j}:\:k\in\N,\; j<\la^k,\; \til{\theta}^k(M_0)_j = M_0\}
    \end{align*}
    is indeed a (finite) group. As each element in our group $G$ is of finite order, all we need to show is that the multiplication of two elements $\sigma^{(k_1)}_{M_0,j_1}, \sigma^{(k_2)}_{M_0,j_2}$, where $\til{\theta}^{k_i}(M_0)_{j_i} = M_0$, is again an element of the set.
    A simple computation yields
  \begin{align*}
    \widehat{\Theta}^{k_2+k_1}(\sigma, M_0)_{j_2 + \la^{k_2}j_1} &= \widehat{\Theta}^{k_2}(\widehat{\Theta}^{k_1}(\sigma,M_0)_{j_1})_{j_2} = \widehat{\Theta}^{k_2}(\sigma \circ \sigma^{(k_1)}_{M_0,j_1}, M_0)_{j_2}\\
    & = (\sigma \circ (\sigma^{(k_1)}_{M_0,j_1} \circ \sigma^{(k_2)}_{M_0,j_2}), M_0),
  \end{align*}
  so the result follows.
  \end{proof}

  Thus, for any $g \in G$, we have some $k_g\geq1$ and $ j_g < \la^{k_g}$ such that $g = \sigma^{(k_g)}_{M_0,j_g}$ and $\til{\theta}^{k_g}(M_0)_{j_g} = M_0$.
  Now, we can finally show the following result.

\begin{Prop}\label{le:primitive}
  $\widehat{\Theta}$ is primitive.
\end{Prop}
\begin{proof}
  As we have $\widehat{\Theta}(\sigma,M_0)_0 = (\sigma,M_0)$, it is sufficient to show that for all $g,h \in G$ and $M,M' \in \mathcal{X}$ there exist $k_1\geq1$, $j_1<\la^{k_1}$ such that $\widehat{\Theta}^{k_1}(id, M_0)_{j_1} = (g, M)$ and $k_2\geq1$, $j_2<\la^{k_2}$ such that $\widehat{\Theta}^{k_2}(h, M')_{j_2} = (id, M_0)$.
  Indeed, for all $k\geq k_1 + k_2$, we obtain that
  \begin{align*}
    \widehat{\Theta}^k(h,M')_{j_2 \la^{k-k_2} + j_1} &= \widehat{\Theta}^{k_1}(\widehat{\Theta}^{k-k_1-k_2}(\widehat{\Theta}^{k_2}(h,M')_{j_2})_{0})_{j_1}\\
    &= \widehat{\Theta}^{k_1}(\widehat{\Theta}^{k-k_1-k_2}(id, M_0)_{0})_{j_1} = \widehat{\Theta}^{k_1}(id, M_0)_{j_1} = (g,M).
  \end{align*}
  As there are only finitely many $g,h \in G$ and $M, M' \in \mathcal{X}$, we find that there exists some (large) $k$ such that for all $g,h \in G$ and $M, M' \in \mathcal{X}$ there exists $j$ such that $\widehat{\Theta}^{k}(h, M')_j = (g, M)$.

  Fix now $g \in G$ and $M \in \mathcal{X}$. Since $g=\sigma^{(k_g)}_{M_0,j_g}$, where $\til{\theta}^{k_g}(M_0)_{j_g} = M_0$, we have (using Lemma~\ref{l:kaem})
  \begin{align*}
    \widehat{\Theta}^{k_{M} + k_g}(id, M_0)_{j_{M} + \la^{k_M} j_g} = \widehat{\Theta}^{k_M}(\widehat{\Theta}^{k_g}(id, M_0)_{j_g})_{j_M} = \widehat{\Theta}^{k_M}(g, M_0)_{j_M} = (g, M).
  \end{align*}
  Furthermore, we have (using~\eqref{eq:hat_prolongable})
  \begin{align*}
    \widehat{\Theta}^{k_{h^{-1}} + k_0}(h, M')_{j_{h^{-1}} + \la^{k_{h^{-1}}} j_0} &= \widehat{\Theta}^{k_{h^{-1}}}(\widehat{\Theta}^{k_0}(h,M')_{j_0})_{j_{h^{-1}}} = \widehat{\Theta}^{k_{h^{-1}}}(h,M_0)_{j_{h^{-1}}}\\
    &= (h \circ h^{-1}, M_0) = (id, M_0),
  \end{align*}
  which completes the proof.
\end{proof}

\begin{Prop}
  $(X_{\til{\Theta}}, S)$ is a topological factor of $(X_{\widehat{\Theta}},S)$.
\end{Prop}
\begin{proof}
  We note that a point in the space $X_{\widehat{\Theta}}$ is of the form $(\sigma_n, M_n)_{n\in \Z}$ and define the equivariant map
  \begin{align}\label{ekwi}
    (\sigma_n, M_n) \mapsto (i_n, M_n),
  \end{align}
  where $i_n$ is defined by $\sigma_n(i_n) = 0$.
  We claim that the fixed point of $\widehat{\Theta}$ that starts with $(id, M_0)$ is mapped to the fixed point of $\til{\Theta}$ that starts with $(0, M_0)$:
Indeed, in view of~\eqref{wzornak}, we have for all $k\geq 0, n<\la^k$ that $\widehat{\Theta}^{k}(id, M_0)_{n} = (\sigma_{M_0, n}^{(k)}, \widetilde{\theta}^k(M_0)_n)$ and, by~\eqref{eq:connection}, we have that $\sigma^{(k)}_{M_0, n}(i) = 0$ if and only if $\widetilde{\Theta}^k(0, M_0)_n = (i, \widetilde{\theta}^k(M_0)_n)$.
  Thus, the map~\eqref{ekwi}, sends the fixpoint of $\widehat{\Theta}$ starting with $(id, M_0)$ to the fixpoint of $\widetilde{\Theta}$ starting with $(0, M_0)$.
Now, the proof follows by the minimality of the dynamical systems under consideration.
\end{proof}
As $(X_{\widetilde\Theta},S)$ is topologically isomorphic to $(X_\Theta,S)$, using Remark~\ref{r:factor1}, we obtain the following.
\begin{Cor}\label{c:topfactor10}
  $(X_{\theta}, S)$ is a topological factor of $(X_{\hat{\Theta}},S)$.
\end{Cor}

We have seen that $h(\theta) = h(\Theta) = h(\widetilde{\Theta})$, and $c(\theta) = c(\Theta) = c(\til{\Theta})$. These equalities do not carry over to $\hat{\Theta}$.
\begin{Lemma}\label{l:cngext}
  We have $c(\hat{\Theta}) = |G|$.
\end{Lemma}
\begin{proof}
  We recall that, by~\eqref{eq:synch}, there are $k_0, j_0$ such that
  \begin{align*}
    \hat{\Theta}^{k_0}(\sigma, M)_{j_0} = (\sigma, M_0),
  \end{align*}
  which shows that $c(\hat{\Theta}) \leq |\{(\sigma,M_0):\:\sigma\in G\}|=|G|$.

  Furthermore, it follows from the definition of $\widehat{\Theta}$ that for all $M \in \mathcal{X}$, we have
  \begin{align*}
    \hat{\Theta}^{k}(\{(g,M): g \in G\})_j = \{(g \cdot \sigma^{(k)}_{M,j}, \widetilde{\theta}^k(M)_j): g \in G\} = \{(g', \widetilde{\theta}^k(M)_j): g' \in G\}.
  \end{align*}
  This shows directly that
  \begin{align*}
    \hat{\Theta}^{k}(G \times \mathcal{X})_j = G \times \widetilde{\theta}^{k}(\mathcal{X})_j,
  \end{align*}
  which implies that $c(\hat{\Theta}) \geq |G|$.
\end{proof}

Even though for the height one could expect equality, it is not the case in general. The following example shows that $h(\hat{\Theta})$ and $h(\theta)$ can be different.

\begin{Example}\label{ex:bij}
We consider the following bijective (whence $c(\theta)=3$) substitution $\theta$:
  \begin{align*}
    \theta(a) &= aab\\
    \theta(b) &= bcc\\
    \theta(c) &= cba.
  \end{align*}
Since $a\mapsto aab$, obviously, $h(\theta)=1$. Of course $\mathcal{X}=\{\{a,b,c\}\}$, so in fact for $\widehat{\Theta}$ we are interested only in the first coordinate. The bijections given by the columns yield three permutations of $\{0,1,2\}$: $\sigma_0=Id$, $\sigma_1=(12)$ and $\sigma_2=(012)$. As they generate $\mathcal{S}_3$, we have  $G=\mathcal{S}_3$.
We also have
  \begin{align*}
    \widehat{\Theta}(\sigma) = (\sigma)(\sigma \circ (12)) (\sigma \circ (021)),
  \end{align*}
  for all $\sigma \in \mathcal{S}_3$.

We consider the fixpoint $u$ that starts with $id$ and find that $u[4] = (12)(12) = id$ and $u[26] = (021)(021)(021) = id$. Thus, see~\eqref{wyso}, we have  $g_0 = {\rm gcd }(S_0) \leq 2$.
  Furthermore, we find that for $n = \sum_{i=0}^{r} n_i 3^i$,
  \begin{align*}
    u[n] = \widehat{\Theta}^{r+1}(id)_n = \sigma_{n_r} \circ \ldots \circ \sigma_{n_0}.
  \end{align*}
  We consider the sign of a permutation and find that $${\rm sgn}(\sigma_0) = {\rm sgn}(\sigma_2) = 1, {\rm sgn}(\sigma_1) = -1$$ which gives
  \begin{align*}
    {\rm sgn}(u[n]) = \prod_{i=0}^{r} {\rm sgn}(\sigma_{n_i}) = \prod_{i=0}^{r} (-1)^{n_i}  = \prod_{i=0}^{r} (-1)^{n_i3^i} = (-1)^{\sum n_i 3^i} = (-1)^n.
  \end{align*}
  This shows that $h(\widehat{\Theta}) = 2$.
\end{Example}

This shows that $\til{\Theta}$ is much closer to $\theta$ than $\widehat{\Theta}$, but $\widehat{\Theta}$ has a structural advantage, as it relies on a group $G$. Indeed, we are able to find a representation of the height within the group $G$.

\begin{Lemma}\label{le:G_0}
  Let $\widehat{\Theta}: G \times \mathcal{X} \to (G \times \mathcal{X})^{\la}$, be as described above. Then, there exists $G_0 \leq G$ such that $G/G_0 \cong \Z/h(\widehat{\Theta})\Z$.
  Furthermore, if we define $\zeta: G/G_0 \times \mathcal{X} \to (G/G_0 \times \mathcal{X})^{\la}$ by
  \begin{align*}
    \zeta(G_0 g, M)_j = (G_0 g \sigma_{M,j}, \widetilde{\theta}(M)_j),
  \end{align*}
  then $\zeta$ is the joining of a periodic substitution $p$ (of period $h$) and $\widetilde{\theta}$, i.e. $\zeta = p \vee \widetilde{\theta}$.
\end{Lemma}
\begin{proof}
  Since $\widehat{\Theta}$ has height $h = h(\widehat{\Theta})$, there exists $f: G\times \mathcal{X} \to \{0,\ldots,h-1\}$ (which we also treat as a 1-code) such that the fixed point $u$ of $\widehat{\Theta}$ obtained by iterations of $\widehat{\Theta}$ at $(id, M_0)$ is mapped via $f$ to $01\ldots(h-1)01\ldots$, i.e. $f(u[n]) = n \bmod h$.
  Furthermore, we can assume without loss of generality that $\lambda \equiv 1 \bmod h$, as we can always replace $\widehat{\Theta}$ by $\widehat{\Theta}^t$, if necessary.
  We find similarly to~\eqref{rowno1} that
  \begin{align*}
    f(\widehat{\Theta}^k(\sigma, M)_j) \equiv f((\sigma, M)) \lambda^k + j \bmod h
  \end{align*}
  for each $k\geq0$ and $j<\lambda^k$.
  Using additionally~\eqref{eq:hat_synchronizing}, this gives
  \begin{align*}
    f((\sigma, M_0)) = f(\widehat{\Theta}^{k_0}(\sigma, M)_{j_0} \equiv f((\sigma, M)) \lambda^{k_0} + j_0 \equiv f((\sigma, M)) + j_0 \bmod h,
  \end{align*}
  which shows that $f((\sigma, M))$ only depends on $\sigma$. Since (for each $k$) $u$ begins with $\widehat{\Theta}^k(id,M_0)$, by the definition of $f$, we have
  $f(\sigma^{(k)}_{M_0,j}) = j \bmod h$. 
  Due to Lemma~\ref{le:Grep}, we can write each $g_1, g_2 \in G$ as $g_i = \sigma_{M_0, j_i}^{(k_i)}$.
  Thus, we find
  \begin{align*}
    f(g_1 g_2) = f(\widehat{\Theta}^{k_2}(\widehat{\Theta}^{k_1}(id, M_0)_{j_1})_{j_2}) \equiv f(\widehat{\Theta}^{k_1}(id, M_0)_{j_1}) + j_2 \equiv j_1 + j_2 \bmod h.
  \end{align*}
  This shows, that for all $g_1, g_2 \in G$ we have
  \begin{align*}
    f(g_1 g_2) \equiv f(g_1) + f(g_2) \bmod h.
  \end{align*}
  It follows that $f$ is a group homomorphism and therefore $G_0:=f^{-1}(0)$ is a normal subgroup.
  We can identify $G g$ by $f(g)$ and the last statement follows immediately for $p: \{0,\ldots,h-1\} \to \{0,\ldots,h-1\}^{\la}$ defined by $p(i)_j = \la i+j \bmod h$.
\end{proof}

\begin{Remark}
  Lemma~\ref{le:G_0} allows us to find a representation of the maximal equicontinuous factor of $\widehat{\Theta}$ in the case that $\widetilde{\theta}$ is not periodic. In general, we need to rely on a more complicated construction.
\end{Remark}

First, we would like to note that we can find a representation of $G$ in the centralizer of $(X_{\widehat{\Theta}}, S)$. Indeed,
given $\tau\in G$, consider
$V_\tau:G \times \mathcal{X}\to
G \times \mathcal{X}$ defined by
\beq\label{gext1}
V_\tau(\sigma, M)=(\tau\circ\sigma,M).\eeq
Note that in view of~\eqref{dejFINAL0}, $\widehat\Theta$ ``commutes''
with $V_\tau$. It follows from Remark~\ref{r:bijcent} that $V_\tau$ (uniquely)
extends to a homeomorphism $V_\tau$ of $X_{\widehat\Theta}$ and
commutes with $S$:
\beq\label{gext111}
V_\tau((\sigma_n,M_n)_{n\in\Z})=(\tau\circ\sigma_n,M_n)_{n\in\Z}\eeq
for each $(\sigma_n,M_n)_{n\in\Z}\in X_{\widehat\Theta}$.
It follows that we have a (finite) group
$\mathcal{V}:=\{V_{\tau}:\tau\in G\}$ of
homeomorphisms of $X_{\widehat\Theta}$ commuting with the shift.

Next, we want to determine the dynamical system obtained by factoring $(X_{\widehat{\Theta}},S)$ by $\mathcal{V}$. For this aim, consider the map $g: (G\times \mathcal{X})^{\Z} \to (G\times \mathcal{X}\times \mathcal{X})^{\Z}$ defined by
\begin{align*}
  g((\sigma_n, M_n)_{n\in \Z}) = (\sigma_{n+1}^{-1} \circ \sigma_n, M_n, M_{n+1})_{n\in \Z},
\end{align*}
which means that $g$ is a sliding block code with (right-) radius $1$.
We see directly that $g((\sigma_n, M_n)_{n\in \Z}) = g((\sigma'_n, M'_n)_{n\in \Z})$ if and only if $(\sigma'_n, M'_n)_{n \in \Z} = V_{\tau}((\sigma_n, M_n)_{n\in \Z})$ for some $\tau \in G$.
We want to describe $g(X_{\widehat{\Theta}})$.
As $g$ is a sliding block code, it follows that $g(X_{\hat{\Theta}})$ is a (minimal) subshift. We want to show that it is actually a substitutional dynamical system, where the substitution has column number $1$, i.e. it is synchronizing.

To this end, let us define a substitution $\eta: G \times \mathcal{X} \times \mathcal{X} \to (G \times \mathcal{X} \times \mathcal{X})^{\la}$ as follows
\begin{align*}
  \eta(\sigma, M, M')_j =
    \left\{\begin{array}{cl}
      (\sigma_{M,j+1}^{-1} \circ \sigma_{M,j}, \til{\theta}(M)_j, \til{\theta}(M)_{j+1}), & \mbox{for } j< \la -1\\
      (\sigma_{M',0}^{-1} \circ \sigma \circ \sigma_{M,\la-1}, \til{\theta}(M)_{\la-1}, \til{\theta}(M')_{0}), & \mbox{for } j = \la-1.
    \end{array}\right.
\end{align*}
A simple computation gives
$$
  \eta^k(\sigma, M, M')_j =
$$$$    \left\{\begin{array}{cl}
      ((\sigma^{(k)}_{M,j+1})^{-1} \circ \sigma^{(k)}_{M,j}, \til{\theta}^{k}(M)_j, \til{\theta}^{k}(M)_{j+1}), & \mbox{for } j< \la^k -1\\
      ((\sigma^{(k)}_{M',0})^{-1} \circ \sigma \circ \sigma^{(k)}_{M,\la-1}, \til{\theta}(M)_{\la-1}, \til{\theta}(M')_{0}), & \mbox{for } j = \la^k-1.
    \end{array}\right.
$$

We denote the fixpoint of $\widehat{\Theta}$ that starts with $(id, M_0)$ by $u$ and denote $u[n] = (\sigma_n, M_n)$.
This allows us to reduce the alphabet of the substitution $\eta$ from $G \times \mathcal{X} \times \mathcal{X}$ to $\mathbb{B}$, where
\begin{align*}
  \mathbb{B} := \{(\sigma_{n+1}^{-1} \circ \sigma_n, M_n, M_{n+1}): n\geq 0\}.
\end{align*}

The formula above shows that $g$ maps the fixed point of $\widehat{\Theta}$ starting with $(\sigma_0, M_0)(\sigma_1, M_1)\ldots$ to the fixed point of $\eta$ starting with $(\sigma_1^{-1} \circ \sigma_0, M_0, M_1)$.
Furthermore, we see directly that the restriction of $f: X_{\widehat{\Theta}} \to \mathbb{B}^{\Z}$ is still well-defined on the fixedpoint of $\widehat{\Theta}$ and therefore on $X_{\widehat{\Theta}}$. This restriction is necessary to ensure that $\eta$ is primitive.

Thus we showed that $g(X_{\widehat{\Theta}}) = X_{\eta}$.

It remains to show that $\eta$ has column number $1$.
As $\til{\theta}$ has column number~$1$, we find by~\eqref{lk2} some integers $k\geq1$, $j<\la^k-1$ and $M'\in \mathcal{X}$ such that $\til{\theta}^{k}(M)_j = M'$ holds for all $M \in \mathcal{X}$.
It follows that for any $\sigma \in G$ and $M_1, M_2 \in \mathcal{X}$, we have
\begin{align*}
  \eta^{2k}(\sigma, M_1, M_2)_{j \la^k + j} &= \eta^k(\eta^k(\sigma, M_1, M_2)_j)_j\\
  &= \eta^{k}((\sigma_{M_1, j+1}^{(k)})^{-1} \circ \sigma_{M_1, j}^{(k)}, M', \theta^{k}(M_1)_{j+1})_j\\
  &= ((\sigma_{M',j+1}^{(k)})^{-1} \circ \sigma_{M',j}^{(k)}, M', \theta^{k}(M')_{j+1})
\end{align*}
which shows that $c(\eta) = 1$.

\begin{Example}
  We give the explicit construction of $\eta$ for the Rudin-Shapiro sequence, which was already introduced in Example~\ref{ex:ex2}.
  We recall that,
  \begin{align*}
    \widetilde{\Theta}(0,M_0) = (0,M_0)(0,M_1), \quad &\widetilde{\Theta}(1,M_0)
    = (1,M_0)(1,M_1),\\
    \widetilde{\Theta}(0,M_1) = (0,M_0)(1,M_1),
    \quad &\widetilde{\Theta}(1,M_1) = (1,M_0)(0,M_1).
  \end{align*}
  This gives directly
  \begin{align*}
    \widehat{\Theta}(\sigma, M_0) &= (\sigma, M_0)(\sigma, M_1),\\
    \widehat{\Theta}(\sigma, M_1) &= (\sigma, M_0)(\sigma \circ (01), M_1).
  \end{align*}
  We see that $\mathbb{B} = \{(\sigma, M, M') \in G \times \mathcal{X} \times \mathcal{X}: M \neq M'\}$ as the only fixpoint of $\widetilde{\theta}$ is given by $M_0M_1M_0M_1\ldots$
  Hence, $\eta$ is defined as follows:
  \begin{align*}
    \eta(\sigma, M_0, M_1) &= (id, M_0, M_1)(\sigma, M_1, M_0),\\
    \eta(\sigma, M_1, M_0) &= ((01), M_0, M_1)(\sigma \circ (01), M_1, M_0).
  \end{align*}
  A simple computation yields $\eta^2(\sigma, M, M')_0 = (id, M_0, M_1)$ for all $(\sigma, M, M') \in \mathbb{B}$ which shows that $\eta$ has column number $1$.
\end{Example}

This procedure works well to find a representation of the maximal equicontinuous factor when $h(\widehat{\Theta}) = 1$. However, in general we are factoring out too much and need to take a subgroup of $\mathcal{V}$, namely $\mathcal{V}_1 := \{V_{\tau}: \tau \in G_0\}$.
Therefore we define $g_h: (G \times \mathcal{X})^{\Z} \to (G \times \mathcal{X} \times \mathcal{X} \times \{0,\ldots,h-1\})^{\Z}$ as an "extension" of the previously mentioned function $g: (G \times \mathcal{X})^{\Z} \to (G \times \mathcal{X} \times \mathcal{X})^{\Z}$ in the following way (see the proof of Lemma~\ref{le:G_0} for $f$ and its properties):
\begin{align*}
  g_h((\sigma_n, M_n)_{n\in \Z}) = (\sigma_{n+1}^{-1} \circ \sigma_n, M_n, M_{n+1}, f(\sigma_n))_{n\in \Z}.
\end{align*}
We see directly that $g((\sigma_n, M_n)_{n\in \Z}) = g((\sigma'_n, M'_n)_{n\in \Z})$ if and only if $(\sigma'_n, M'_n)_{n \in \Z} = V_{\tau}((\sigma_n, M_n)_{n\in \Z})$ for some $\tau \in G_0$, as $f(\tau \sigma) \equiv f(\tau) + f(\sigma) \bmod h$.

Thus, we consider now $\eta_h: G \times \mathcal{X} \times \mathcal{X} \times \{0,\ldots,h-1\} \to (G \times \mathcal{X} \times \mathcal{X} \times \{0,\ldots,h-1\})^{\la}$, where the first three coordinates coincide with $\eta$ and the last coordinate can be seen as another substitution $p: \{0,\ldots,h-1\} \to \{0,\ldots, h-1\}^{\la}$ where $p(i)_j = \la i + j \bmod h$.

This gives that $\eta_h = \eta \vee p$.
As we have seen that $X_{\eta}$ is measure-theoretically isomorphic to $(H_{\lambda}, m_{H_{\lambda}}, R)$ and $(X_p, S) \cong (\Z/ h\Z, m_{\Z/ h\Z}, \tau_h)$, we have found a factor of $X_{\widehat{\Theta}}$ that is itself given by $X_{\eta_h}$ and is measure-theoretically isomorphic to the maximal equicontinuous factor of $X_{\widehat{\Theta}}$.

\section{Proof of Theorem~\ref{t:main}, Corollary~\ref{c:main} and some questions}

{\em Proof of Theorem~\ref{t:main}}: We have already shown~(i) of Theorem~\ref{t:main}. We now need to handle the case $c(\theta)>h(\theta)$.

To this end, we have defined $\theta\vee\widetilde{\theta}$  extending $\theta$ (and being primitive), where
$(X_{\widetilde\theta},S)$ ``represents'' (measure-theoretically) either the $(H_\lambda, m_{H_\lambda},R)$ factor of
$(X_\theta,S)$ if $\theta$ is not quasi-bijective or it is finite. Via an isomorphic copy $\widetilde\Theta$ of $\theta\vee\widetilde\theta$, we finally have got the substitution $\widehat\Theta$ which is primitive by Proposition~\ref{le:primitive}. Moreover, by Corollary~\ref{c:topfactor10},
$(X_{\widehat{\Theta}},S)$  has  $(X_\theta,S)$ as its
topological factor, and therefore $c(\widehat\Theta)>h(\widehat\Theta)$ (it cannot be synchronized).~\footnote{It has also $(X_{\widetilde\theta},S)$ as its topological factor.}. We only need to prove~(ii) of Theorem~\ref{t:main} for $(X_{\widehat\Theta},S)$.

Via \eqref{gext111}, we have shown that there is a compact group, a copy of the group $G$, namely $\mathcal{V}=\{V_\tau:\:\tau\in G\}$, in the centralizer of $S$ which consists of homeomorphisms. Because of~Lemma~\ref{l:cngext} and Proposition~\ref{p:cextension}, the factor $\sigma$-algebra $\cb(\mathcal{V})$ of subsets which are fixed by all elements of $\mathcal{V}$ ``represents'' (measure-theoretically) $(H_\lambda,m_{H_{\lambda}},R)$.
Of course, the $L^2$-space for
$(X_{\widehat\Theta},\mu_{\widehat\Theta},S)$ is spanned by
continuous functions. Now, the conditional expectation of $F\in C(X_{\widehat\Theta})$ with respect to $\cb(\mathcal{V})$ is given by
$\frac1{|G|}\sum_{\tau\in G}F\circ V_\tau$, so it is still a continuous function. It follows that each continuous function $F$ is
represented as $F_1+F_2$, where $F_1$ is continuous and orthogonal to the
$L^2(\cb(\mathcal{V}))$, and $F_2$ is represented by a
continuous function that belongs to $L^2(\cb(\mathcal{V}))$. If we knew that $H_\lambda$ represents the Kronecker factor of $(X_{\widehat\Theta},\mu_{\widehat\Theta},S)$, via Proposition~\ref{p:essc3}, we could apply Section~\ref{s:mainstrategy} to conclude the proof. But, as we have already noticed, we cannot control $h=h(\widehat\Theta)$ and the Kronecker factor $\mathcal{K}$ of $(X_{\widehat\Theta},\mu_{\widehat\Theta},S)$ is given by $(H_\lambda\times \Z/h\Z,m_{H_\lambda}\otimes m_{\Z/h\Z},R\times\tau_h)$. But $\cb(\mathcal{V})\subset\mathcal{K}$, so by a result of Veech~\cite{Ve}, there is a compact subgroup $\mathcal{V}_1\subset\mathcal{V}$ such that $\mathcal{K}$ is the $\sigma$-algebra of subsets fixed by all elements of $\mathcal{V}_1$.
(We are actually able to give a concrete description of $\mathcal{V}_1$ as  $\{V_{\tau}: \tau \in G_0\}$, where $G_0$ was defined in Lemma~\ref{le:G_0}.) So, we can now repeat all the above arguments (cf.\ Footnote~\ref{f:maintool}) with $\mathcal{V}$ replaced by $\mathcal{V}_1$ to complete the proof. \bez

\noindent
{\em Proof of Corollary~\ref{c:main}}:
We begin with a $1$-code $F:X_\theta\to\C$ (that is, $F(y)$ depends only on the 0-coordinate $y[0]$ of $y\in X_\theta$). We assume that for some fixed point $y\in X_\theta$, the sequence $(F(S^ny))_{n\in \N}$ is multiplicative. Such an $F$ is a continuous function and $(X_\theta,S)$ is a topological factor of $(X_{\widehat\Theta},S)$, say $\pi:X_{\widehat\Theta}\to X_\theta$ settles a factor map. Now, $F\circ\pi\in C(X_{\widehat\Theta})$ and (as in the proof of Theorem~\ref{t:main}) we have $F\circ \pi=F_1+F_2$, where both $F_1,F_2$ are continuous, $F_1\in L^2(\mathcal{B}(\mathcal{V}))$ and $F_2\perp L^2(\mathcal{B}(\mathcal{V}))$. Take $x\in\pi^{-1}(y)$ that can be chosen as a fixed point of $\widehat\Theta$. By Theorem~\ref{t:main},
$$
\frac1N\sum_{n\leq N}\ov{F_2(S^nx)}F(S^n(\pi(x)))\to0$$
since $(F(S^n\pi(x)))$ is multiplicative. On the other hand, by unique ergodicity,
$$
\frac1N\sum_{n\leq N}\ov{F_2(S^nx)}F(S^n(\pi(x)))\to\int_{X_{\widehat\Theta}}\ov{F_2}\cdot F\circ\pi\,d\mu_{\widehat\Theta}.$$
It follows that $F\circ\pi\perp F_2$, which implies $F_2=0$.
It follows that the spectral measure of $F \circ \pi =F_1$ is discrete.

Let us first assume that $h(\widehat{\Theta}) = 1$.
First we note that $\pi$ is actually a $1$-code, which shows that $F\circ \pi$ is also a $1$-code.
Furthermore, we have an explicit formula for $F_1$:
\beq\label{expfor}
  F_1(x) = \frac{1}{|G|} \sum_{\tau \in G} (F\circ \pi) \circ V_{\tau}(x).
\eeq
As $F \circ \pi$ is a $1$-code and $V_{\tau}$ is also a $1$-code, it follows that $F_1$ is a $1$-code as well.
Moreover, by~\eqref{expfor}, we obtain that
\begin{align*}
  F_1(\sigma, M) = \frac{1}{|G|} \sum_{\tau \in G} (F\circ \pi) (\tau \sigma, M) = \frac{1}{|G|} \sum_{\sigma' \in G} (F\circ\pi)(\sigma', M).
\end{align*}
Thus,  $F_1(\sigma, M)$ only depends on $M$ and we see that
\begin{align*}
  F(S^n y) = F_1(S^n \overline{M}) = F_1(M_n).
\end{align*}
As we have shown that $F_1\in C(X_{\widetilde\theta})$ is a $1$-code and $\widetilde\theta$ is synchronizing by Proposition~\ref{p:listas}~(i) and the result follows.

Consider now the case $h(\widehat{\Theta})>1$.
We recall that there exists some $f: G \times \mathcal{X} \to \{0,\ldots,h-1\}$ such that $f(y[n]) \equiv n \bmod h$.
As we have seen at the end of Section~\ref{sec:group}, we can identify $\mathcal{V}_1$ as $\{V_{\tau}: \tau \in G_0\}$.
Now, we find that $F_1(\sigma, M)$ depends on $M$ and $\sigma G_0$, or equivalently on $M$ and $f(\sigma, M) = f(\sigma G_0,M)$.
This gives that
\begin{align*}
  F(S^n y) = F_1((n \bmod h), M_n).
\end{align*}
As $(M_n)_{n\in \N}$ and $(n \bmod h)_{n\in \N}$ are both Weyl-rationally almost periodic, we see directly that $(F_1((n \bmod h), M_n))_{n\in \N}$ is also Weyl-rationally almost periodic.
Furthermore, we define a substitution$p:\{0,\ldots,h-1\}\to\{0,\ldots,h-1\}^\lambda$, by setting:
\begin{align*}
  p(i)_j = i \la +j \bmod h,
\end{align*}
for all $j<\la$ and $i\in\{0,\ldots,h-1\}$. We immediately compute that $c(p) = h(p) = h$.
Moreover, we find that $((n \bmod h), M_n))_{n\in \N}$ is the fixpoint of $p \vee \widetilde{\theta}$.
We find by Lemma~\ref{le:joining_height} that
\begin{align*}
  \max(c(p), c(\widetilde{\theta})) \leq c(p \vee \widetilde{\theta}) \leq c(p) c(\widetilde{\theta})
\end{align*}
and since $c(\widetilde{\theta})=1$ this gives $c(p \vee \widetilde{\theta}) = h$.
Furthermore, we have
\begin{align*}
  {\rm lcm}(h(p), h(\widetilde{\theta})) | h(p \vee \widetilde{\theta})
\end{align*}
and as $h(\widetilde{\theta}) = 1$ this gives that $h \leq h(p \vee \widetilde{\theta})$.
However, Lemma~\ref{le:joining_height0} implies that $h(p \vee \widetilde{\theta}) \leq c(p \vee \widetilde{\theta})$ and therefore $h(p \vee \widetilde{\theta}) = h$.
It follows that $(F(S^n(y))_{n\in \N}$ is given as a $1$-code of a fixpoint of a substitution of constant length for which the column number equals the height, or in other words, it has purely discrete spectrum.
\bez

\begin{Remark}
If we want to obtain in the assertion a weaker conclusion,\footnote{Note that $\mob^2$ is Besicovitch rationally almost periodic but is not Weyl rationally almost periodic, e.g.\ \cite{Be-Ku-Le-Ri}.} namely that such functions are Besicovitch rationally almost periodic, another proof of Corollary~\ref{c:main} can be obtained in the following way.

We first recall that if $$\cm:=\{\bfv:\N\to\C:\:|\bfv|\leq1\text{ and }\bfv\text{ is multiplicative}\}$$
then to check that  $\bfv$ has a mean along an arbitrary arithmetic progression, it is enough to show that $\bfv\cdot\chi$ has a mean for each Dirichlet character $\chi$, see e.g.\ Proposition 3.1 in \cite{Be-Ku-Le-Ri2}. Now, the argument used
in the proof of Lemma~2 in \cite{Sch-P} works well and it shows that a function $\bfv\in\cm$ taking finitely many values has a mean along any arithmetic progression. Furthermore, Theorem~1.3 in \cite{Be-Ku-Le-Ri2} says that $\bfv$ is either Besicovitch rationally almost periodic or it is uniform. But uniformity is equivalent to aperiodicity by \cite{Fr-Ho3}, and we have obtained an aperiodic multiplicative automatic sequence, which is in conflict with Theorem~\ref{t:main}.
\end{Remark}

\section{Proof of Corollary~\ref{c:main2}}\label{s:lastS}
Let $\bfu:\N\to\C$ be a bounded multiplicative function. Following \cite{Ab-Ku-Le-Ru}, we say that a topological dynamical system $(Y,T)$ satisfies the strong $\bfu$-MOMO property\footnote{Instead of the strong $\mob$-MOMO property, we speak about the strong MOMO property. The acronym MOMO stands for {\em M\"obius Orthogonality of Moving Orbits.}} if for each increasing sequence $(b_k)_{k\geq1}\subset\N$, $b_1=1$, $b_{k+1}-b_k\to\infty$, each sequence $(y_k)\subset Y$ and each $f\in C(Y)$, we have
\beq\label{zalo1}
\lim_{K\to\infty}\frac1{b_K}\sum_{k\leq K}\left|
\sum_{b_{k}\leq n<b_{k+1}}f(T^ny_k)\bfu(n)\right|=0\eeq
Substituting $f=1$ above, we see that $\bfu$ has to satisfy~\eqref{zalo2}.

Clearly, given $(Y,T)$, \eqref{zalo1} implies \eqref{mdis}. In fact (for $\bfu=\mob$), in the class of zero entropy systems, they are equivalent: it is proved in \cite{Ab-Ku-Le-Ru} that Sarnak's conjecture holds if and only if all zero entropy systems enjoy the strong MOMO property.\footnote{Moreover, no positive entropy systems has the strong MOMO property. We recall (see \cite{Do-Se}) that there are positive entropy systems $(Y,T)$ satisfying~\eqref{mdis} for
all $f\in C(Y)$.}

\begin{Lemma}\label{l:momo1} For each primitive substitution $\theta$, the system $(X_\theta,S)$ has the strong $\bfu$-MOMO property for each bounded, aperiodic,
multiplicative $\bfu:\N\to \C$  satisfying~\eqref{zalo2}.\end{Lemma}
\begin{proof} We only need to prove the result for $(X_{\widehat\Theta},S)$.  Then, the result follows immediately from Theorem~25 in \cite{Ab-Ku-Le-Ru} or, more precisely, from its proof. Indeed, all ergodic rotations satisfy the strong $\bfu$-MOMO property (Corollary~28 in \cite{Ab-Ku-Le-Ru}) and the only ergodic joinings between $S^p$ and $S^q$  whenever $p\neq q$ are large enough, are relative products over isomorphisms of Kronecker factors. Moreover, the structure of the space of continuous functions allows us to represent each $F\in C(X_{\widehat\Theta})$ as $F=F_1+F_2$, with both $F_i$ continuous, $F_1$ measurable with respect to the Kronecker factor, and $F_2$ orthogonal to $L^2$ of that factor. Finally, we apply the reasoning from the proof of Theorem~25 to $F_2$.\end{proof}

\noindent
{\em Proof of Corollary~\eqref{c:main2}}
It follows from Lemma~\ref{l:momo1} and the equivalence of Properties~1 and~3 in Main Theorem in~\cite{Ab-Ku-Le-Ru} that all MT-substitutional systems satisfy the strong $\bfu$-MOMO property for each bounded, aperiodic, multiplicative $\bfu$ satisfying~\eqref{zalo2}. The uniform convergence follows now from Theorem~7 in \cite{Ab-Ku-Le-Ru}.\bez

Note that among MT-substitutional systems there are uniquely ergodic models which are topologically mixing \cite{Lehrer}. In such models the maximal equicontinuous factor must be trivial, that is, there is no topological ``realization'' of the Kronecker factor.

Theorem~\ref{t:main} tell us that for each primitive substitution $\theta$, each bounded, multiplicative, aperiodic $\bfu$,~\eqref{mdis} holds for each $f\in C(X_\theta)$ and arbitrary $x\in X_\theta$. We do not know however whether this convergence is uniform in $x$. Note that Corollary~\ref{c:main2} gives the positive answer if $\bfu$ satisfies additionally~\eqref{zalo2}.

\bibliographystyle{plain}

\vspace{3ex}

Faculty of Mathematics and Computer Science, Nicolaus Copernicus University, Chopin street 12/18, 87-100 Toru\'n, Poland: mlem@mat.umk.pl

Institut f\"ur Diskrete Mathematik und Geometrie, TU Wien,  Wiedner Hauptstr. 8-10, 1040 Wien, Austria: clemens.muellner@tuwien.ac.at
\end{document}